\documentclass{article}

\PassOptionsToPackage{utf8}{inputenc}
\usepackage{inputenc}

\usepackage{amsfonts}			
\usepackage{amsmath}			
\usepackage{amssymb}			
\usepackage{amsthm}				
\usepackage[ruled, Algorithmus]{algorithm}  
\usepackage[noend]{algpseudocode}
\usepackage{algorithmicx}
\usepackage{aligned-overset} 
\usepackage{bbm}          
\usepackage{blindtext} 			 
\usepackage[labelsep=period,justification=justified]{caption}
\usepackage{csquotes}
\usepackage[dvipsnames]{xcolor}	
\usepackage{enumitem,xparse}   
\usepackage{float}				
\usepackage[a4paper, margin=1in]{geometry}
\usepackage{graphicx}			
\usepackage{mathrsfs}			
\usepackage{mathtools}			
\usepackage{parskip}      
\usepackage{pgfplots}
\usepackage{subcaption}			
\usepackage{thmtools}       
\usepackage{tikz}         
\usetikzlibrary{arrows}
\usetikzlibrary{arrows.meta}  
\usetikzlibrary{arrows.meta}  
\usetikzlibrary{backgrounds} 
\usetikzlibrary{calc} 
\usetikzlibrary{decorations.pathreplacing}  
\usetikzlibrary{fit} 
\usetikzlibrary{patterns}
\usetikzlibrary{positioning}
\usetikzlibrary{tikzmark}
\usepackage{url}				
\usepackage{xparse}     
\usepackage{mathabx}

\usepackage[english]{babel} 

\usepackage{hyperref}
\usepackage[capitalise]{cleveref}  

\crefname{lemma}{Lemma}{Lemmata}
\crefname{subsection}{Subsection}{Subsections}

\let\olditemize=\itemize
\let\endolditemize=\enditemize
\renewenvironment{itemize}{\olditemize \itemsep0em}{\endolditemize}

\newtheoremstyle{aussagen}
{}                
{}                
{\slshape}        
{}                
{\bfseries}       
{}                
{0.5em}           
{\thmname{#1}\thmnumber{ #2}. \thmnote{(#3)}}  

\newtheorem{theorem}{Theorem}[section]
\newtheorem{lemma}[theorem]{Lemma}

\newtheorem{definition}[theorem]{Definition}
\newtheorem{proposition}[theorem]{Proposition}
\newtheorem{remark}[theorem]{Remark}
\theoremstyle{definition}
\newtheorem{example}{Example}

\DeclareUnicodeCharacter{2212}{\textendash}

\definecolor{darkblue}{RGB}{40,40,120}
\definecolor{darkred}{RGB}{220,20,20}
\definecolor{jgurot}{RGB}{193,0,42}
\definecolor{jgublue}{RGB}{0,138,193}
\colorlet{red}{jgurot}
\colorlet{blue}{jgublue}

\let\pgfimageWithoutPath\pgfimage 
\renewcommand{\pgfimage}[2][]{\pgfimageWithoutPath[#1]{bildchen/#2}}

\def\todos{1}
\if\todos1
\newcommand{\todo}[1]{\textcolor{Maroon}{{\bf [todo:} #1]}}
\else
\newcommand{\todo}[1]{}
\fi

\def\comments{1}
\if\comments1
  \newcommand{\cmnt}[1]{\textcolor{Green}{{[Comment: #1]}}}
\else
  \newcommand{\cmnt}[1]{}
\fi

\def\mycomments{0}
\if\mycomments1
  \newcommand{\mycmnt}[1]{\textcolor{TealBlue}{{[My Comment: #1]}}}
\else
  \newcommand{\mycmnt}[1]{}
\fi

\newcommand{\sss}{\scriptscriptstyle}
\usepackage{nicefrac}

\DeclarePairedDelimiterXPP\Prob[1]{\mathbb{P}}{(}{)}{}{
  \providecommand\given{\nonscript\:\delimsize\vert\nonscript\:\mathopen{}}
  #1}

\DeclarePairedDelimiterXPP\Exp[1]{\mathbb{E}}{[}{]}{}{
  \providecommand\given{\nonscript\:\delimsize\vert\nonscript\:\mathopen{}}
  #1}

\DeclarePairedDelimiterXPP\Var[1]{\mathrm{Var}}{[}{]}{}{
  \providecommand\given{\nonscript\:\delimsize\vert\nonscript\:\mathopen{}}
  #1}

\DeclarePairedDelimiterXPP\Cov[1]{\mathrm{Cov}}{[}{]}{}{
  \providecommand\given{\nonscript\:\delimsize\vert\nonscript\:\mathopen{}}
  #1}

\let\deg\relax
\DeclarePairedDelimiterXPP\deg[2]{\mathrm{deg}_{#1}}{(}{)}{}{#2}
\DeclarePairedDelimiterXPP\degpl[2]{\mathrm{deg}^+_{#1}}{(}{)}{}{#2}

\DeclarePairedDelimiterX\skp[2]{\langle}{\rangle}{#1, #2}

\DeclarePairedDelimiter\abs{\lvert}{\rvert}

\DeclarePairedDelimiter\floor{\lfloor}{\rfloor}

\DeclarePairedDelimiterX\Set[1]{\{}{\}}{
  \providecommand\given{\nonscript\:\delimsize\vert\nonscript\:\mathopen{}}
  #1}

\let\originalleft\left
\let\originalright\right
\renewcommand{\left}{\mathopen{}\mathclose\bgroup\originalleft}
\renewcommand{\right}{\aftergroup\egroup\originalright}

\newcommand{\R}{\mathbb{R}}
\newcommand{\N}{\mathbb{N}}

\newcommand{\1}{\mathbbm{1}}

\newcommand{\norm}[2]{\|#1\|_{#2}}

\let\ds\displaystyle

\makeatletter
\renewcommand{\nsim}{\mathrel{\mathpalette\n@sim\relax}}
\newcommand{\n@sim}[2]{%
  \ooalign{%
    $\m@th#1\sim$\cr
    \hidewidth$\m@th#1\rotatebox[origin=c]{60}{$#1-$}$\hidewidth\cr
  }%
}
\makeatother

\makeatletter
\newcommand*\rel@kern[1]{\kern#1\dimexpr\macc@kerna}
\renewcommand*\widebar[1]{%
  \begingroup
  \def\mathaccent##1##2{%
    \rel@kern{2.2}%
    \overline{\rel@kern{-1.8}\macc@nucleus\rel@kern{0.1}}%
    \rel@kern{-0.4}%
  }%
  \macc@depth\@ne
  \let\math@bgroup\@empty \let\math@egroup\macc@set@skewchar
  \mathsurround\z@ \frozen@everymath{\mathgroup\macc@group\relax}%
  \macc@set@skewchar\relax
  \let\mathaccentV\macc@nested@a
  \macc@nested@a\relax111{#1}%
  \endgroup
}
\makeatother

\makeatletter
\renewcommand*\env@matrix[1][*\c@MaxMatrixCols c]{%
  \hskip -\arraycolsep
  \let\@ifnextchar\new@ifnextchar
  \array{#1}}
\makeatother

\let\epsilon\varepsilon

\newcommand\addtag{\refstepcounter{equation}\tag{\theequation}}

\begin{document}

\title{Multi-drawing P\'olya urns via labelled random DAGs}
\author{Cécile Mailler \& Rebecca Steiner}
\date{August 28, 2025}
\maketitle

\begin{abstract}
  A P\'olya urn of replacement matrix $R=(R_{i,j})_{1\leq i,j\leq d}$ is a Markov process that encodes the following experiment: an urn contains balls of $d$ different colours and at every time-step, a ball is drawn uniformly at random in the urn, and if its colour is $i$, then it is replaced in the urn with an additional $R_{i,j}$ balls of colour $j$, for all $1\leq i, j\leq d$.
  We study a natural extension of this model in which, instead of drawing one ball at each time-step, we draw a set of $m\geq 2$ balls: in this case, the replacement matrix becomes a replacement \emph{tensor}. 
  Because of the multi-draws, this process can no longer be seen as a branching process, which makes its analysis much more intricate than in the classical P\'olya urn case. Partial results proved by stochastic approximation techniques exist in the literature.
  In this article, we introduce a new approach based on seeing the process as a stochastic process indexed by a random directed-acyclic graph (DAG) and use this approach, together with the theory of \emph{stochastic tensors}, to prove a convergence theorem for these multi-drawing P\'olya urns, with assumptions that are straightforward to check in practice.
\end{abstract}

\section{Introduction}
A $d$-colour P\'olya urn of replacement matrix $\mathfrak r = (\mathfrak r_{i,j})_{1\leq i,j\leq d}$ is a stochastic process that describes the contents of an urn containing balls of $d$ different colours. 
At every time step, we pick a ball uniformly at random, and if its colour is~$i$, then we replace it in the urn together with an additional $\mathfrak r_{i,j}$ balls of colour $j$, for all $1\leq j\leq d$.
The case when the replacement matrix is the identity (or an integer multiple of it) dates back to Markov~\cite{Markov} as well as P\'olya \& Eggenberger~\cite{EP23}. In this case, if $U(n) = (U_1(n), \ldots, U_d(n))$, where $U_i(n)$ is the number of balls of colour $i$ in the urn at time $n$, then $U(n)/n$ converges almost surely to a random variable whose distribution depends on $U(0)$.
A very different behaviour is observed when the replacement matrix $\mathfrak r$ is irreducible: in this case, by Perron-Frobenius theorem, the spectral radius $\lambda$ of $\mathfrak r$ is a simple eigenvalue of $\mathfrak r$, and it admits a unique eigenvector $v$ whose coordinates are all non-negative and such that $\|v\|_1 = 1$. 
Athreya \& Karlin~\cite{AK68} proved that, in this case, almost surely as $n\uparrow\infty$, $U(n)/n \to \lambda v$. 
Contrary to the ``identity case'', the limit here is deterministic and does not depend on $U(0)$.
Since Athreya \& Karlin's work, the model has been largely studied and many generalisations have been considered: for example, Janson \cite{Jans04} generalises the convergence result of Athreya and Karlin to the case when the replacement matrix is random and balls of different colours have different weights/activities, and provides functional limit theorems for the fluctuations of $U(n)/n$ around its limit.

The aim of this paper is to generalise the results of Athreya \& Karlin~\cite{AK68} to the context of multi-drawing urns: the idea is that, instead of drawing one ball at random in the urn at every time step, we draw a couple of balls from the urn (with replacement) and, if the first ball is of colour $j$, and the second one of colour $k$, then we add $R_{ijk}$ balls of colour $i$, for all $1\leq i\leq d$.
The collection of integers $(R_{ijk})_{1\leq i,j,k\leq d}$ is called the ``replacement tensor'' 
of the multi-drawing urn.
Note that our aim is in fact to generalise to the case when an $m$-tuple of balls is drawn from the urn at every time step, and our results do apply to $m>2$.
However, in order to keep notation as simple as possible we choose to state and prove our results in the case $m=2$; the extension to $m>2$ is outlined in~\cref{sec:multi}.
The main difficulty is that the methods of Athreya \& Karlin~\cite{AK68} and Janson~\cite{Jans04} both rely on embedding the urn process into continuous time to get a multi-type branching process. 
Although this could be done in the multi-drawing case, the resulting continuous-time process is not a branching process. 

Multi-drawing P\'olya urns have been studied in the literature, mostly the two-colour case.
For example, Chen \& Wei~\cite{CW}, Chen \& Kuba~\cite{CK} and Kuba, Mahmoud \& Panholzer~\cite{KMP}
and Aguech, Lasmar \& Selmi~\cite{ALS} all study particular examples.
Kuba \& Mahmoud~\cite{KM} and Kuba \& Sulzbach~\cite{KS}
gave general convergence results (and up to third-order asymptotics) for the composition of a $2$-colour urn under the assumption that there exists two deterministic sequences $(\alpha_n)_{n\geq 0}$ and $(\beta_n)_{n\geq 0}$ such that, 
for all $n\geq 0$, $\mathbb E[U_1(n)|\mathcal F_n] = \alpha_n U_1(n) +\beta_n$; they call this assumption the ``affinity'' assumption. This assumption is arguably quite restrictive, but it has the advantage to make the problem amenable to standard martingale techniques.
Later, Lasmar, Mailler \& Selmi~\cite{LMS18} proved a convergence result for general $m$-drawing P\'olya urns. Their results, proved using stochastic approximation methods, apply without the affinity assumption, but under assumptions that are difficult to verify in practice.

In this paper, we use a different point of view on multi-drawing P\'olya urns to provide general convergence results for multi-drawing P\'olya urns under assumptions that are easy to verify in practice.
Namely, inspired by the work of Mailler \& Marckert~\cite{MM17} who see P\'olya urns 
as branching Markov chains on the random recursive tree, 
we see the multi-drawing P\'olya urns as ``higher order Markov chains'' or ``$m$-dependent Markov chains'' indexed by random DAGs. 
Our proof relies on both the theory of higher-order Markov chains and their transition tensors (equivalent of the transition matrix for a Markov chain), and asymptotic results on the shape of a typical large random DAG.

In the rest of the introduction, we first state our main result, before giving the main ideas behind the proof, and then discussing the result in light of the existing literature.

\subsection{Main Result}

\begin{definition}\label{def:multi-urn}
  Let $(R_{ijk} = R(i,j,k))_{1\leq i,j,k\leq d}$ be a collection of non-negative real numbers.
  We define the \emph{2-drawing d-colour Pólya urn with replacement $R$} as a stochastic process $(U(n) = (U_1(n), \ldots, U_d(n)))_{n\geq 0}$ on $\mathbb N^d$ satisfying
  \[\qquad U_i(n+1) = U_i(n) + R(i, C_1(n), C_2(n)),\]
  for all $n\geq 0$ and $1\leq i\leq d$,
  where $(C_1(n), C_2(n))$ is a couple of independent and identically-distributed random variables on $\{1, \ldots, d\}$ satisfying
  \[\Prob*{C_1(n) = k}=\Prob*{C_2(n) = k} = \frac{U_k(n)}{\sum_{i=1}^d U_i(n)},\] 
  for all $k \in \{1,\ldots,d\}$.
\end{definition}
In other words, if we interpret $(U(n))_{n\geq 0}$ as the composition of an urn as time increases, then this process encodes the following experiment: At every time-step, we pick two balls uniformly at random in the urn and let $C_1(n)$ denote the colour of the first ball, $C_2(n)$ the colour of the second ball (note that we put the first ball back into the urn before drawing the second ball, and then put the second ball back in the urn as well). We then add $R(i,C_1(n), C_2(n))$ balls of colour $i$ in the urn, for all $1\leq i\leq d$.

\begin{remark}
  In the interpretation as balls in an urn, we need that $R_{ijk}\in\mathbb N$ for all $1\leq i,j,k\leq d$. 
  This assumption is not needed in Definition~\ref{def:multi-urn}. 
  Also note that we require that $R_{ijk}\geq 0$ for all $1\leq i,j,k\leq d$; in terms of balls in an urn, this means that we never remove balls from the urn. 
  In practice, the model still makes sense if we remove balls from the urn, as long as all the coordinates of $U(n)$ stay non-negative. 
  In the literature on classical P\'olya urns, one can either work conditionally on the event that all coordinates of $(U(n))_{n\geq 0}$ stay non-negative, or choose a replacement rule such that this event holds almost surely.
\end{remark}

Note that, just as a $d\times d$ matrix defines a linear map from $\mathbb R^d$ into itself, a collection $(R_{ijk} = R(i,j,k))_{1\leq
  i,j,k\leq d}$ of real numbers defines a bilinear map from $\mathbb R^d\times \mathbb R^d$ into $\mathbb R^d$. Indeed, for all $x = (x_1, \ldots, x_d)$ and $y = (y_1, \ldots, y_d)$ in $\mathbb R^d$, if we set, for all $1\leq k\leq d$,
\begin{equation}\label{eq:def_tensor}
  \mathcal R_k(x, y) = \sum_{i=1}^d \sum_{j=1}^d R_{kij} x_i y_j,
\end{equation}
and $\mathcal R(x,y) = (\mathcal R_1(x,y), \ldots, \mathcal R_d(x,y))\in \mathbb R^d$, 
then $\mathcal R$ is a bilinear map from $\mathbb R^d\times R^d\to \mathbb R^d$. 
We refer to both $R=(R_{ijk})_{1\leq i,j,k\leq d}$ and $\mathcal R$ as ($3$-mode) tensors.
Note that, if we let $(e_1, \ldots, e_d)$ be the canonical basis of $\mathbb R^d$, then, for all $1\leq i,j\leq d$,
\[\mathcal R(e_i, e_j) = (R_{1ij}, \ldots, R_{dij}),\]
which correspond to what we add in the urn when we first draw a ball of colour $i$ and then a ball of colour $j$.

In this paper, we assume that the replacement tensor $(R_{ijk})_{1\leq i, j, k\leq d}$ satisfies the following assumptions:
\begin{itemize}
\item[{\bf (T)}] For all $1\leq i,j,k\leq d$, $R_{ijk} \geq 0$.
\item[{\bf (B)}] There exists $\sigma \in (0,\infty)$ such that, for all $1\leq j,k\leq d$, $\sum_{i=1}^d R_{ijk} = \sigma$.
\item[{\bf (E)}] $ \max_{1\leq j, j', k, k'\leq d} \sum_{i=1}^d \abs*{R_{i j k} - R_{i j' k'}} < \sigma$.
\end{itemize}

We define
\[\Delta_d = \Set*{x \in [0,1]^d \given \norm{x}{1} = 1}.\]

\begin{theorem}\label{thm:main-result}
  Let $(U(n))_n$ be a $2$-drawing $d$-colour urn whose replacement tensor $\mathcal R$ satisfies assumptions {\bf (T)}, {\bf (B)} and {\bf (E)}. Then, 
  \begin{itemize}
  \item[{\bf (i)}] there exists a unique $x^* \in \Delta_d$ such that $\sigma x^* = \mathcal R(x^*,x^*)$ and,
  \item[{\bf (ii)}] as long as $\|U(0)\|_1>0$, in probability as $n\uparrow\infty$,
    \[\frac{U(n)}{\|U(n)\|_1} \to x^*.\] 
  \end{itemize}
\end{theorem}

\begin{remark}[Discussion of the assumptions]
  The first assumption ensures that the urn is ``tenable'', i.e.\ that all coordinates of $(U(n))_{n\geq 0}$ are non-negative at all times, and thus the urn process is well-defined.
  The second assumption is common (although not desirable) in the classical P\'olya urn case and is called the ``balance'' assumption. 
  It ensures that the total number of balls in the urn is a deterministic function of time; 
  more precisely, for all $n\geq 0$,
  \[\|U(n)\|_1 = \sum_{i=1}^d U_i(n) = \|U(0)\|_1 + \sigma n.\]
  Assumption {\bf (E)} ensures that the map $x\mapsto \mathcal R(x,x)/\sigma$ is a contraction: it implies that there exists $q\in(0,1)$ such that, for all $x,y\in\Delta_d$,
  $\|\mathcal R(x, x)- \mathcal R(y,y)\|_1\leq q\sigma\|x-y\|_1$. 
  In the classical case of balanced uni-drawing urns (of replacement matrix $\mathfrak r$), 
  this assumption would be $\max_{1\leq i,j\leq d}\mathfrak r_{i,j}<\sigma$, which, as
  explained in~\cite[Theorem~4.3]{FT20}, 
  implies that the map $x\mapsto x\mathfrak r/\sigma$ is a contraction and thus that the Markov chain
  of transition matrix $\mathfrak r/\sigma$ is ergodic. 
  This is why we call this assumption the \emph{ergodicity} assumption.
\end{remark}

A straightforward generalisation is for the case of (finitely-many-colour) $m$-drawing P\'olya urns. The only reason we state our main theorem and its proof for $m=2$ is to keep the notation as simple as possible and thus make the result and its proof easier to read. No additional idea is needed for larger values of~$m$.
The equivalent of Theorem~\ref{thm:main-result} and a sketch of its proof are provided in Section~\ref{sec:multi}.

Interestingly, most of our proof holds even if the set of colours is infinite (possibly uncountable).
However, Assumption {\bf (E)} is much more abstract in this setting and difficult to check in practice: 
we thus only state our result for infinitely-many-colour, 
two-drawing P\'olya urns in Section~\ref{sec:infinite}.

\subsection{Ideas of the proof}
As mentioned in the introduction, our proof relies on seeing the $m$-drawing P\'olya urn as an ``$m$-dependent'' Markov chain indexed by a random DAG. 
For all $n\geq 0$, node $n$ in the DAG is labelled by the set of balls added at time $n$ in the urn.
The $m$ parents of node $n$ are the nodes corresponding to the $m$ balls drawn at random at time~$n$ in the urn. We will show that the DAG obtained is the ``uniform recursive DAG'', and that the distribution of the label of each node only depends on the labels of its~$m$ parents.
When $m=1$, as seen in~\cite{MM17}, we get a branching Markov chain on the random recursive tree; we call the analogous object when $m\geq 2$ an ``$m$-dependent Markov chain on the uniform recursive DAG''.

We first use results of Janson~\cite{Jans14} to prove that the uniform recursive DAG ``looks locally like a complete $m$-ary tree'' and thus reduce our study to $m$-dependent Markov chains on the infinite $m$-ary tree. This removes the difficulty coming from the fact that a DAG has (non-directed) loops.

Ergodicity of $m$-dependent Markov chains on the infinite $m$-ary tree has been studied in the literature (see, e.g.~Ching, Ng and Fung~\cite{CNF08}). As developed by Lim~\cite{Lim05} and Qi~\cite{Qi05}, 
the transition probabilities of such a process can be represented by a transition tensor (instead of a transition matrix for classical $1$-dependent Markov chains).
The question of existence and uniqueness of stationary distributions for $m$-dependent Markov chains leads to the study of fixed points, or \emph{Z-eigenvectors}, of such tensors.
As tensors are multilinear maps, 
solving fixed-point equations is in general much harder than for matrices (in fact, it is NP-hard for general tensors -- see~\cite{SPY16,Cir07}). 
Some sufficient criteria for uniqueness of a fixed point of a transition tensor are given in Li and Ng~\cite{LN14}, Chang and Zhang~\cite{CZ13}, Wang and Li~\cite{WL23}, Fascino and Tudisco~\cite{FT20} (see also~\cite{CQZ13} for an early survey). 
A more general view is taken in Gautier, Tudisco and Hein~\cite{GTH23,GT19,GTH19,GTH19.2}, where the authors generalise the Perron-Frobenius theory to multi-homogeneous mappings.

Further, the existence of a unique fixed point, or stationary distribution, of such a tensor is not sufficient to imply convergence (in distribution) of the $m$-dependent Markov chain.
To get convergence of our $m$-dependent Markov chain on the infinite $m$-ary tree, we apply results of Fascino and Tudisco~\cite{FT20}.

\subsection{Discussion of our main result in view of the existing literature}
As already mentioned, general results for the convergence of multi-drawing P\'olya urns already exist in the literature. The aim of this section is to discuss our main result in view of this existing literature.

One strand of the literature focuses on ``affine'' P\'olya urns: see, in particular, Kuba \& Mahmoud~\cite{KM} and Kuba \& Sulzbach~\cite{KS}. 
They say that a P\'olya urn $(U(n))_{n\geq 0}$ 
is affine if there exists two deterministic sequences $(\alpha_n)_{n\geq 0}$ and $(\beta_n)_{n\geq 0}$ such that, 
for all $n\geq 0$, $\mathbb E[U_1(n)|\mathcal F_n] = \alpha_n U_1(n) +\beta_n$.
They are able to prove up to third-order asymptotic results for urns that satisfy this assumption.
In Section~\ref{sub:affine}, we show how to apply Theorem~\ref{thm:main-result} 
to affine two-drawing urns and recover the first-order results of Kuba \& Mahmoud~\cite{KM} 
for a large class of affine urns (albeit not all of them, importantly).
Because our result is not tailored for affine P\'olya urns, but applies more generally, it should not be surprising that our result does not fully recover the first-order results of Kuba \& Mahmoud~\cite{KM}, which use powerful martingale methods that apply only to affine urns.
These martingale methods are also the reason why Kuba \& Mahmoud~\cite{KM} and Kuba \& Sulzbach~\cite{KS} are able to prove second- and third-order asymptotic results for affine urns, while these remain completely open in our more general setting.

A different general result for multi-drawing P\'olya urns is given by Lasmar, Mailler and Selmi~\cite{LMS18}: 
they are able to prove up-to-second-order asymptotic results for balanced multi-drawing P\'olya urns.
The drawback of their first-order result is that the assumptions are difficult to check on practical examples; the advantage of it is its generality. 
In Section~\ref{sub:tunisian}, we apply Theorem~\ref{thm:main-result} to examples treated in~\cite{LMS18}. On one of these examples, our main result applies and gives convergence of the composition of the urn with much less effort than in~\cite{LMS18}. 
In many other of these examples, our main result simply does not apply.
In conclusion, the main advantage of Theorem~\ref{thm:main-result} is that its assumptions are easy to apply; however, these assumptions seem to be quite restrictive in practice.

Note that our theorem only applies to drawing balls ``with replacement'': when drawing our set of $m$ balls at every time-step, we first draw the first ball, take note of its colour, put it back in the urn, then draw the second ball, etc.
The results of~\cite{LMS18} hold both with or without replacement; those of \cite{KM} and \cite{KS} hold without replacement (although their proofs would adapt straightforwardly to the with replacement case).
Because the total number of balls in the urn tends deterministically to infinity (by Assumption {\bf (B)}), 
drawing with or without replacement {\it should not make any difference at large times}\footnote{This heuristic is supported by results in the literature where distances between the two measures are estimated (see, e.g.~\cite{Freedman77, Stam78}).} 
and our result should hold also in the without-replacement case. 
However, our proof does not straightforwardly extend to the without-replacement case.

Also note that, in our framework, at every step, the order in which the $m$ balls are sampled
can be taken into account. 
In other word, we do not assume that $R_{ijk} = R_{ikj}$ for al $1\leq i,j,k\leq d$.
As far as we know, this is not a possibility in the existing literature mentioned so far (\cite{KM, KS, LMS18}).
In Section~\ref{sub:affine}, we include two examples in which the order of the $m$ balls matter in the replacement rule and call them ``asymmetric examples''.

Finally, and importantly, a strength of our result and proof is that it extends to infinitely-many colour P\'olya urns, as we show in Section~\ref{sec:infinite}. The only caveat is that the equivalent of Assumption {\bf (E)} for infinitely-many colour P\'olya urns is very hard to check, mainly because the paper of Fasino and Tudisco~\cite{FT20} only holds for $m$-dependent Markov chains on finite state-space: 
we leave it for future work to extend the results of Fasino and Tudisco~\cite{FT20} to infinite state-space and thus provide an easier-to-check sufficient assumption for {\bf (E)}.

{\bf Plan of the paper:}
In Section~\ref{sec:MC}, we introduce our new point of view on multi-drawing P\'olya urns: seeing them as a random labelling of a uniform random directed acyclic graph (DAG). We also give preliminary results on the typical shape of the uniform random DAG.
In Section~\ref{sec:distribution-xn}, we prove Theorem~\ref{thm:main-result}.
In Section~\ref{sec:multi}, we show how Theorem~\ref{thm:main-result} can be adapted to larger values of~$m$.
In Section~\ref{sec:examples}, we give a series of examples.
Finally, in Section~\ref{sec:infinite}, we show how our model and Theorem~\ref{thm:main-result} can be extended to multi-drawing P\'olya urns with infinitely-many colours.

\section{A Markov chain indexed by a uniform random DAG model}\label{sec:MC}
In this section, we show that the multi-drawing P\'olya urn can be seen as a randomly-labelled uniform recursive DAG. This is the equivalent to the result of Mailler \& Marckert \cite{MM17}, which says that P\'olya urns can be seen as branching Markov chains on the random recursive tree.
Because of the multi-drawing, the underlying graph is a DAG instead of a tree.

\subsection{Seeing the urn process as a measure-valued process}\label{sub:MVPP}
In the following, it is convenient to consider the measure-valued process
\[\bigg(m_n := \sum_{i=1}^d U_i(n) \delta_i\bigg)_{n\geq 0},\]
which contains all the information about $(U(n))_{n\geq 0}$.
Under assumption {\bf (T)}, $m_n$ is a measure on $\{1, \ldots, d\}$ for all $n\geq 0$.
For all $n\geq 0$, we let
\[\hat m_n = \frac{m_n}{m_n(\{1, \ldots, d\})},\]
so that $\hat m_n$ is a probability measure for all $n\geq 0$.

With this new perspective, it is convenient to also have a measure version of the replacement tensor: for all $1\leq j,k\leq d$, we let
\begin{equation}\label{eq:replacement_measures}
  \rho(j,k) = \sum_{i=1}^d R_{ijk}\delta_i.
\end{equation}
With this notation, we have that, for all $n\geq 0$,
\[m_{n+1} = m_n + \rho(C_1(n), C_2(n)),\]
where, given $m_n$, $(C_1(n), C_2(n))\sim \hat m_n\otimes \hat m_n$.

\begin{lemma}\label{lem:truc}
  The following two statements are equivalent:
  \begin{itemize}
  \item[{\bf (a)}] $U(n)/\|U(n)\|_1 \to x^*$ in probability as $n\uparrow\infty$.
  \item[{\bf (b)}] $\hat m_n \to \nu$ in probability as $n\uparrow\infty$, for the topology of weak convergence on the set of probability measures on $\{1, \ldots, d\}$, where $\nu = \sum_{i=1}^d x_i^* \delta_i$.
  \end{itemize}
\end{lemma}

\begin{remark}
  Let $(\mu_n)_{n\geq 0}$ and $\mu$ be probability measures on a measurable space $E$.
  We say that $\mu_n\to \mu$ for the topology of weak convergence if, for all continuous and measurable functions $f\colon E\to \mathbb R$, $\int_E f \mathrm d\mu_n \to \int_E f \mathrm d\mu_n$.
  If $E = \{1, \ldots, d\}$, it is equivalent to $\mu_n(\{i\})\to \mu(\{i\})$ for all $1\leq i\leq d$, and thus Lemma~\ref{lem:truc} is immediate.
\end{remark}

Our aim is thus to show that $\hat m_n \to \nu = \sum_{i=1}^d x_i^*\delta_i$, in probability as $n\uparrow\infty$. To do so, we use the following lemma, which is folklore, and whose proof can be found in~\cite{MM17} (see Lemma~3.1 therein). We use the following notation: for any two probability measures $\mu$ and $\nu$, we write $(A,B)\sim \mu\otimes \nu$ to say that $A\sim\mu$ and $B\sim\nu$ are independent.

\begin{lemma}\label{lem:randmeas-conv}
  Let $(\nu_n)_{n\geq 0}$ be a sequence of random probability measures on a measurable space $E$ and $\nu$ be a deterministic probability measure on $E$. 
  For any integer $n\geq 0$, given $\nu_n$, we let $(A_n, B_n)\sim\nu_n\otimes\nu_n$.
  If, in distribution as $n\uparrow\infty$,
  \[(A_n, B_n) \Rightarrow \nu\otimes \nu,\]
  then $\nu_n \to \nu$ in probability as $n\uparrow\infty$, for the topology of weak convergence on the set of probability measures on~$E$.
\end{lemma}

\begin{remark}
  In Lemma~\ref{lem:randmeas-conv}, it is important to note that $A_n$ and $B_n$ are not independent random variables. Indeed they both depend on $\nu_n$, their random distribution.
\end{remark}

\subsection{Coupling of the urn process with a randomly labelled DAG}\label{sub:coupling}

\begin{definition}
  For fixed $\pi$ a probability measure on $\{1, \ldots, d\}$ 
  and $3$-mode tensor~$\mathcal R$, which satisfies Assumptions {\bf (T)} and {\bf (B)},
  we jointly define a sequence of random DAGs $(G_n = (V_n, E_n))_{n\geq 0}$, and a sequence of random labels $(X(n) = (X_1(n), X_2(n))\in\{1, \ldots, d \}^2)_{n\geq 1}$ as follows:
  We first set $G_0 = (\{0\}, \varnothing)$ (this is the graph with one node and no edges). 
  Node $0$ has no label.
  Then, for all $n\geq 0$, given $G_n$ and $X(0), \ldots, X(n)$,
  we let $V_{n+1} = \{0, \ldots, n+1\}$ be the set of nodes of $G_{n+1}$ and:
  \begin{itemize}
  \item We sample $M_{n+1}$ and $D_{n+1}$ in $\{0, \ldots, n\}$ independently, uniformly at random in $\{0, \ldots, n\}$.
  \item We then sample $\tilde C_1(n)$ and $\tilde C_2(n)$ independently with the following distributions:
    \begin{itemize}
    \item If $M_{n+1} = 0$, then $\tilde C_1(n)\sim \pi$, otherwise, $\tilde C_1(n)\sim \rho(X(M_{n+1}))/\sigma$.
    \item If $D_{n+1} = 0$, then $\tilde C_2(n)\sim \pi$, otherwise, $\tilde C_2(n)\sim \rho(X(D_{n+1}))/\sigma$.\end{itemize}
    (See Equation~\eqref{eq:replacement_measures} for the definition of $\rho$.)
  \end{itemize}
  We let $E_{n+1} = E_n \cup \{(n+1, M_{n+1})\}\cup\{(n+1, D_{n+1})\}$ 
  be the set of edges of $G_{n+1}$. 
  In other words, at every time-step, we add a node to the DAG, and create two directed
  edges from this new node to its two randomly-chosen parents 
  (the notation $M(n)$ and $D(n)$ stand for ``Mum'' and ``Dad'').
  
  We let, for all $1\leq i\leq d$,
  \[X(n+1) = (\tilde C_1(n), \tilde C_2(n)),\]
  where we interpret $X(n)$ as the ``label of node $n$''. 
  By Assumptions {\bf (T)} and {\bf (B)}, $\rho(X(n))/\sigma$ is a probability distribution.
  Thus, the process $(G_n, X(n))_{n\geq 0}$ is well-defined; 
  we call this process the $\mathcal R$-labelled DAG process with initial distribution $\pi$.
\end{definition}

\begin{lemma}
  Let $\mathcal R$ be a tensor satisfying Assumptions {\bf (T)} and {\bf (B)} and
  let $(G_n, X(n))_{n\geq 0}$ be the $\mathcal R$-labelled DAG process with initial distribution $\pi$.
  For all $n\geq 0$, let $m_n = \pi + \sum_{i=1}^n \rho(X(n))$.
  Then, $(m_n)_{n\geq 0}$ is the (measure-valued version of the) multi-drawing P\'olya urn with initial composition $\pi$ and replacement tensor~$\mathcal R$.
\end{lemma}

\begin{proof}
  By induction.
\end{proof}

\begin{lemma}
  Let $R$ be a tensor satisfying Assumptions {\bf (T)} and {\bf (B)} and
  let $(G_n, X(n))_{n\geq 0}$ be the $\mathcal R$-labelled DAG process of initial distribution $\pi$.
  Then $(G_n)_{n\geq 0}$ is the uniform recursive DAG.
\end{lemma}

\begin{proof}
  Immediate.
\end{proof}

The advantage of this coupling is the following fact: for all $n\geq 0$, given $(G_n, X(0), \ldots, X(n))$,
\[\big(\tilde C_1(n), \tilde C_2(n)\big)\sim \hat m_n\otimes \hat m_n.\]
Thus, by Lemma~\ref{lem:randmeas-conv}, to prove Theorem~\ref{thm:main-result}(ii), it is enough to show that, in distribution as $n\uparrow\infty$,
\begin{equation}\label{eq:vcC1C2}
  X(n+1) = \big(\tilde C_1(n), \tilde C_2(n)\big) \Rightarrow \nu \otimes \nu,
\end{equation}
where $\nu = \sum_{i=1}^d x^*_i \delta_i$.

In other words, the following theorem implies Theorem~\ref{thm:main-result}$(ii)$:
\begin{theorem}\label{thm:convergence}
Let $R$ be a tensor satisfying Assumptions {\bf (T)}, {\bf (B)} and {\bf (I)}.
Let $(G_n, X(n))_{n\geq 0}$ be the $\mathcal R$-labelled DAG process.
Then,
in distribution as $n\uparrow\infty$,
\[X(n)\Rightarrow \nu\otimes \nu.\]
\end{theorem}

The proof of Theorem~\ref{thm:convergence} is done in two independent parts:
In the first part, we prove that nodes $M_{n}$ and $D_{n}$ are typically quite deep in $G_n$ (i.e.\ far from node $0$) and that, in a neighbourhood of node $n$, the DAG is locally a complete binary tree; 
this is done using results of
Janson~\cite{Jans14} on the uniform recursive DAG. 
In the second part, we use the theory of $Z$-eigenvectors of tensors to prove convergence
in distribution of $(X(M_{n}), X(D_{n}))$.

\subsection{Shape of the uniform recursive DAG}\label{sec:shape-gn}
The aim of this section is to understand the genealogy of node $n$ in the uniform random DAG.
We show that up to distance $o(\log\log\log n)$, this genealogy is a complete binary tree.

We look at the genealogy (set of all ancestors) of node $n$ in the uniform random DAG $G_n$ (i.e.\ its ``mum'' and ``dad'', and by
induction, its grandparents, great-grandparents, etc.) and then look at the graph~$H_n$ induced by the ancestors $v$ of $n$
{such that $v \geq n_1 \coloneqq \lfloor n/\log n\rfloor$ 
(i.e.\ we keep only these nodes and the edges between them, and erase the rest of $G_n$).
Janson~\cite{Jans14} shows that, with high probability, $H_n$ is a (binary) tree.
In this section, we complete this result by showing that, if $\ell(n) = o(\log\log\log n)$, then the tree $H'_n$ obtained by erasing all nodes of~$H_n$ that are at (graph) distance more than $\ell(n)$ of $n$ is a complete binary tree (i.e.\ $n$ has degree~2 and all the nodes at distance at most $\ell(n)-1$ of $n$ have degree~3).
We sketch this result on Figure~\ref{fig:shape}.
We let
\begin{align}
\mathcal E_n &= \{H_n\text{ is a binary tree}\},\notag\\
\mathcal F_n &= \{H'_n \text{ is a complete binary tree}\}.\label{eq:def_EF}
\end{align}
\begin{lemma}\label{lem:EF}
If $\ell(n) = o(\log\log\log n)$, then $\lim_{n\uparrow\infty}\mathbb P(\mathcal E_n\cap \mathcal F_n) = 1$.
\end{lemma}

\begin{figure}
\begin{center}
\includegraphics[width=6cm]{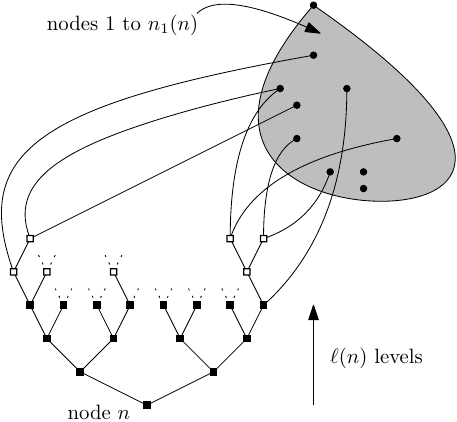}
\end{center}
\caption{A picture of the uniform recursive DAG when $\mathcal E_n\cap\mathcal F_n$ holds. The grey area represents the part of the DAG composed of nodes 1 to $n_1(n) = \lfloor n/\log n\rfloor$ and the edges between them. The nodes outside of the grey area are nodes $n_1(n)+1$ to $n$. The black square nodes are the nodes of tree $H'_n$ (a complete binary tree); the white and black square nodes are the nodes of $H_n$ (a binary tree). The dotted half edges represent edges whose other end is in the grey area - we do not draw them to keep the figure readable.}
\label{fig:shape}
\end{figure}

\begin{proof} By Theorem 3.1 in~\cite{Jans14},  we have $\mathbb P(\mathcal E_n) \to 1$ as $n\uparrow\infty$. In the rest of the proof, we assume that $\mathcal E_n$ occurs.
Recall that, by definition, for all $n\geq 1$, $M(n)$ (and $D(n)$) are two independent, uniform elements of $0, \ldots, n-1$. 
We can model this by sampling a sequence of random variables $(U_n, V_n)_{n\geq 1}$, i.i.d.\ and uniformly distributed on $[0,1]$: for all $n\geq 1$, we set $M(n) = \lfloor U_n n\rfloor$ and $D(n)= \lfloor V_n n\rfloor$.
We let $L(n)$ the set of ancestors of $n$ that are at distance $\ell(n)$ of~$n$.
For any ancestor $u\in L(n)$, in distribution, $u = \floor{\ldots \floor{\floor{nX_1}X_2}\ldots X_{\ell(n)}}$ where $(X_i)_{i\geq 1}$ is a sequence of i.i.d.\ random variables, uniformly-distributed on $[0,1]$. 
And, on the event $\mathcal E_n$, $|L(n)| = 2^{\ell(n)}$.
Thus,
\begin{align}
\Prob*{\exists v \in L_n: v < n_1(n)}
&\leq \sum_{v \in L_n} \Prob*{v < n_1(n)}
=2^{\ell(n)} \Prob*{\floor{\ldots \floor{\floor{nX_1}X_2}\ldots X_{\ell(n)}}<n_1(n)}\notag\\
&\leq 2^{\ell(n)} \Prob*{nX_1\cdots X_{\ell(n)}-\ell(n) < n_1(n)}.\label{eq:truc}
\end{align}
Now note that
\begin{align*}
\Prob*{nX_1\cdots X_{\ell(n)}-\ell(n) < n_1(n)}
&= \mathbb P\bigg(\sum_{i=1}^{\ell(n)} \log (1/{X_i}) > \log\bigg(\frac{n}{n_1(n)+\ell(n)}\bigg)\bigg)\\
&= \mathbb P\bigg(\sum_{i=1}^{\ell(n)} \big(\log (1/X_i)-1\big) > \log\bigg(\frac{n}{n_1(n)+\ell(n)}\bigg)-\ell(n)\bigg).
\end{align*}
Recall that $n_1(n) = \floor{n/\log n}$ and $\ell(n)=o(\log\log n)$. 
Thus, as $n\uparrow\infty$,
\[\log\bigg(\frac{n}{n_1(n)+\ell(n)}\bigg)-\ell(n)
= \log n - \log\bigg(\frac n{\log n}+o(\log \log n)\bigg) + o(\log\log n)
= \log\log n + o(\log\log n).\]
Thus, for all $n$ large enough, this is positive, which implies, by Chebyshev's inequality,
\[\Prob*{nX_1\cdots X_{\ell(n)}-\ell(n) < n_1(n)}
\leq \frac{\ell(n)}{(1+o(1))(\log\log n)^2},\]
because $\mathrm{Var}\big(\sum_{i=1}^{\ell(n)} \log (1/{X_i})\big) = \ell(n)$.
Thus, by~\eqref{eq:truc}, as $n\uparrow\infty$,
\[\Prob*{\exists v \in L_n: v < n_1(n)}
\leq\frac{\ell(n)2^{\ell(n)}}{((1+o(1))\log\log n)^2} \to 0,\]
as $n\uparrow\infty$ because $\ell(n) = o(\log\log\log n)$.
\end{proof}

\section{Random labelling of the DAG}\label{sec:distribution-xn}
In this section, we aim to understand the limit of $X(n)$, the random couple labelling node~$n$ in the $\mathcal R$-labelled DAG, as $n\uparrow\infty$. 
First note that, in the construction of the $\mathcal R$-labelled DAG, one can first sample the unlabelled graph $(G_n)_{n\geq 0}$, and then sample the labels $(X(n))_{n\geq 1}$.

Recall that, as defined in Section~\ref{sec:shape-gn}, 
$H_n$ is the set of all the ancestors $v$ of node~$n$ such that $v \geq n_1$.
Also, $H'_n$ is the sub-graph of $H_n$ in which 
we only keep the nodes that are at graph distance at most $\ell(n)$ of node~$n$.
On $\mathcal E_n\cap \mathcal F_n$, $H'_n$ is a complete binary tree and, in this tree, the label of each node is sampled at random, according
to a distribution that is uniquely determined by the label of its parents:
a ``2-dependent Markov chain''.
We therefore start this section by defining 2-dependent Markov chains on the infinite binary tree, and study their ergodicity.

\subsection{Two-dependent Markov chains on the infinite binary tree}\label{sec:two-dep-mc}
\begin{definition}\label{def:MC}
Let $(\pi_i)_{i\geq 1}$ be a sequence of probability distributions on a finite or countable space $\mathcal S$.
Let $T = (T_{ijk} = T(i,j,k))_{i,j,k \in \mathcal S}$ 
a collection of non-negative real 
numbers satisfying $\sum_{k\in\mathcal S} T_{ijk} = 1$ for all $i,j\in\mathcal S$.
We sample $W^{\sss (0)}_i\sim \pi_i$ for all $i\geq 1$, independently of each other.
We then sample the values of $(W_i^{\sss (n)})_{i\geq 1}$ for all $n\geq 1$ recursively as follows: 
for all $n\geq 1$, for all $i\geq 1$, for all $x, y, w\in \mathcal S$
\[\mathbb P(W_i^{\sss (n)} = w \:\vert\: (W_i^{\sss (n-1)})_{i\geq 1}) 
= T\big(w, W^{\sss (n-1)}_{2i}, W^{\sss (n-1)}_{2i+1}\big).\]
We call this process $(W_i^{\sss (n)})_{i\geq 1, n\geq 0}$ the 2-dependent Markov chain of transition tensor $T$ and initial distribution $(\pi_i)_{i\geq 1}$.
\end{definition}

Note that a 2-dependent Markov chain can be interpreted as a random labelling of the infinite binary tree.
For this interpretation, we index the nodes of the infinite binary tree as follows (see Figure~\ref{fig:bin_tree}): 
for all $n\geq 0$, $i\geq 1$, node $(n,i)$ is the $i$-th node in the $n$-th level.
For all $n\geq 1$, $i\geq 1$, the two parents of node $(n,i)$ are nodes $(n-1, 2i)$ and $(n-1, 2i+1)$.
The value $W_i^{\sss (n)}$ can be interpreted as the label of node $(n,i)$: its distribution given the labels of its two parents is given by $T(\cdot, W_{2i}^{\sss (n-1)}, W_{2i+1}^{\sss (n-1)})$.

\begin{figure}
\begin{center}
\includegraphics[width = 8cm]{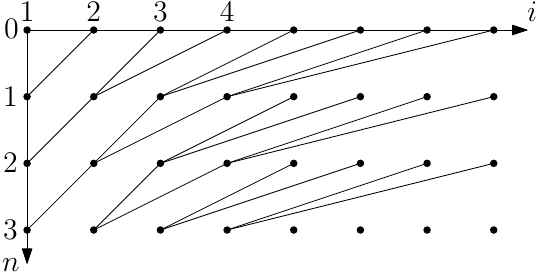}
\end{center}
\caption{The infinite binary tree.}
\label{fig:bin_tree}
\end{figure}

\begin{lemma}\label{lem:xin}
Let $(\pi_i)_{i\geq 1}$ be a sequence of probability distributions on a finite or countable space $\mathcal S$.
Let $T = (T_{ijk} = T(i,j,k))_{i,j,k \in \mathcal S}$ 
a collection of non-negative real 
numbers satisfying $\sum_{k\in\mathcal S} T_{ijk} = 1$ for all $i,j\in\mathcal S$.
Let $(W_i^{\sss (n)})_{i\geq 1, n\geq 0}$ be the 2-dependent Markov chain of transition tensor $T$ and initial distribution $(\pi_i)_{i\geq 1}$. 
For all $n\geq 0, i\geq 1$, we let $\pi_i^{\sss (n)}$ be the distribution of $W_i^{\sss (n)}$.
Then,  for all $i\geq 1$, $\pi_i^{\sss (0)} = \pi_i$ and, for all $n\geq 1$, $i\geq 1$, $w\in\mathcal S$,
\begin{equation}\label{eq:rec_laws}
\pi_i^{\sss (n)} (w) 
= \sum_{x,y\in\mathcal S} T_{wxy} \pi_{2i}^{\sss (n-1)}(x) \pi_{2i+1}^{\sss (n-1)}(y).
\end{equation}
\end{lemma}

Note that, if we let $\mathcal T : \mathcal S\times \mathcal S \to \mathcal S$ be the 3-mode tensor associated to $T =  (T_{ijk})_{i,j,k \in \mathcal S}$ (as we did in~\eqref{eq:def_tensor} for $(R_{ijk})_{1\leq i,j,k\leq d}$),
and if we interpret the distributions $\pi_i^{\sss (n)}$ as vectors whose entries are indexed by $\mathcal S$, then~\eqref{eq:rec_laws} can be written as 
\[\pi_i^{\sss (n)} = \mathcal T\big(\pi_{2i}^{\sss (n-1)}, \pi_{2i+1}^{\sss (n-1)}\big).\]
With this in mind, it is natural to define invariant distributions as follow:
\begin{definition}\label{def:inv_distr}
We say that $\pi = (\pi(w))_{w\in S}$ is an invariant distribution of the 3-mode tensor $\mathcal T$ if and only if $\pi = \mathcal T(\pi, \pi)$.
\end{definition} 

It is now natural to ask under which conditions a 2-dependent Markov chain of transition tensor~$T$ admits an invariant distribution and converges to it. 
To do so,  we introduce the notion of \emph{ergodicity coefficients}~\cite{Sen06}. The ($L^1$-norm) ergodicity coefficient of a stochastic matrix whose rows and columns are indexed by $\mathcal S$ is defined as
\begin{equation}\label{def:tau}
\tau(\mathfrak p) = \sup_{x,y\in\Delta(\mathcal S)} \frac{\|\mathfrak p x - \mathfrak p y\|_1}{\|x-y\|_1},
\end{equation}
where 
\[\Delta(\mathcal S) = \Set{x = (x_i)_{i\in\mathcal S}\in [0,1]^{\abs*{\mathcal S}}\given \sum_{i\in \mathcal S} x_i = 1}.\]
If $\tau(\mathfrak p)<1$, then $x\mapsto \mathfrak px$ admits a unique fixed point $x^*$ in $\Delta(\mathcal S)$ and, for all $x_0\in \Delta(\mathcal S)$, 
if $x_{n+1} = \mathfrak p x_n$ for all $n\geq 0$, then $x_n \to x^*$ as $n\uparrow\infty$. 

To define the ergodicity coefficient of a $3$-mode tensor, Fasino and Tudisco~\cite{FT20} start by defining the following matrices: Given a $3$-mode tensor $T = (T_{ijk})_{i,j,k\in\mathcal S}$ and a vector $x\in\Delta(\mathcal S)$, we let, for all $i,j\in\mathcal S$,
\begin{equation}\label{eq:def_Rlr}
(T^{(\ell)}_x)_{i,j} = \sum_{k=1}^d T_{ikj}x_k \qquad (T^{(r)}_x)_{i,j} = \sum_{k=1}^d T_{ijk}x_k.
\end{equation}
(The superscripts stand for ``left'' and ``right''.)
If $T$ is a stochastic tensor (i.e.\ if $\sum_{i\in\mathcal S} T_{ijk} = 1$ for all $j,k\in\mathcal S$), then $T^{(\ell)}_x$ and $T^{(r)}_x$
are stochastic matrices for all $x\in\Delta{(\mathcal S)}$.
We let
  \begin{equation}\label{eq:def_taus}
  \tau_\ell(T) \coloneqq{} \max_{x \in \Delta(\mathcal S)} \tau(T^{(\ell)}_x)
    \quad\text{ and }\quad 
     \tau_r(T) \coloneqq{} \max_{x \in \Delta(\mathcal S)} \tau(T^{(r)}_x).
     \end{equation}
By~\cite[Thm 5.1, Cor 5.2]{FT20},
\[
  \tau_\ell(T) ={} \frac12 \max_{j,j',k} \sum_i \abs*{T_{ijk} - T_{ij'k}}   \quad\text{ and }\quad 
  \tau_r(T) ={} \frac12 \max_{j,k,k'} \sum_i \abs*{T_{ijk} - T_{ijk'}}
\]
which implies the following result:
\begin{lemma}\label{lem:AssE}
For any stochastic tensor $T = (T_{ijk})_{i,j,k\in\mathcal S}$, the two following statements are equivalent:
\begin{itemize}
\item[\rm (i)] $\ds\max_{j,k,k'\in\mathcal S} \sum_{i\in\mathcal S} \abs*{T_{ijk} - T_{ijk'}} 
+ \max_{j,j',k\in\mathcal S} \sum_{i\in\mathcal S} \abs*{T_{ijk} - T_{ij'k}} < 2$;
\item[\rm (ii)] $\tau_\ell(T) + \tau_r(T) < 1$.
\end{itemize}
\end{lemma}

Finally, the following lemma gives conditions under which a 2-dependent Markov chain admits an invariant distribution, and converges to it.
\begin{proposition}\label{lem:tensor-props}
Let $T = (T_{ijk})_{i,j,k\in\mathcal S}$ be a stochastic tensor satisfying condition {\rm (i)} of Lemma~\ref{lem:AssE}.
Then there exists a unique $\pi\in \Delta(\mathcal S)$ 
such that $\mathcal T(\pi, \pi) = \pi$. 
Furthermore, 
any sequence $(\pi_i^{\sss (n)})_{n\geq 0, i\geq 1}$ of vectors of $\Delta(\mathcal S)$ that satisfies~\eqref{eq:rec_laws} is such that, for all $n\geq 0$,
\[\norm{\pi_1^{\sss (n)}-\pi}{1} \leq q^n \max_{i\geq 1} \norm{\pi_i^{\sss (0)} - \pi}{1},\]
for $q = \tau_\ell(T) + \tau_r(T)<1$.
Therefore, $\lim_{n \to \infty} \pi_1^{\sss (n)} = \pi$.
\end{proposition}

Note that, by Lemma~\ref{lem:xin}, $\lim_{n \to \infty} \pi_1^{\sss (n)} = \pi$ implies that $W_1^{\sss (n)} \Rightarrow \pi$ in distribution as $n\uparrow\infty$, where $(W_i^{\sss (n)})_{n\geq 0, i\geq 1}$ is the 2-dependent Markov chain of transition tensor $T = (T_{ijk})_{i,j,k\in\mathcal S}$.

\begin{proof}
Under condition (i) of Lemma~\ref{lem:AssE}, $\tau_\ell(T) + \tau_r(T) < 1$ implies the existence of a unique fixed point $\pi$ of $\mathcal
T$ via~\cite[Thms 6.1 and 6.3]{FT20}. Further, for all $x, y, u, v \in \Delta(\mathcal S)$, we have
  \begin{align*}
    \mathcal T(x,y) - \mathcal T(u,v)     
    ={}& \mathcal T(x,y) - \mathcal T(x,v) + \mathcal T(x,v) - \mathcal T(u,v) \\
    ={}& \mathcal T(x,y-v) + \mathcal T(x-u,v) 
    =T^{(\ell)}_x(y-v) + T^{(r)}_{v}(x-u).
  \end{align*}
  (See~\eqref{eq:def_Rlr} for the definition of the matrices $T^{(\ell)}_x$ and $T^{(r)}_{v}$; also recall that, by definition, for all $x,y\in\mathbb R^{\mathcal S}$, for all $i\in \mathcal S$, $\mathcal T(x,y)_i = \sum_{j,k\in\mathcal S} T_{ijk} x_jy_k$.)
This implies that, for all $x, y \in \Delta(\mathcal S)$,
  \begin{align*}
    \norm{\mathcal T(x,y) - \pi}{1} ={}
    & \norm{T^{(\ell)}_x(y-\pi) + T^{(r)}_{\pi}(x-\pi)}{1}\\
    \leq{}& \norm{T^{(\ell)}_x(y-\pi)}{1} + \norm{T^{(r)}_{\pi}(x-\pi)}{1} \\
    \leq{}& \tau(T^{(\ell)}_x)\norm{y-\pi}{1} + \tau(T^{(r)}_{\pi})\norm{x-\pi}{1},
    \end{align*}
    by definition of $\tau(\cdot)$ (see~\eqref{def:tau}).
    Now using the definition of $\tau_\ell(T)$ and $\tau_r(T)$ (see~\eqref{eq:def_taus}), we get
    \begin{align}
     \norm{\mathcal T(x,y) - \pi}{1}
    \leq{}& \tau_\ell(T)\norm{y-\pi}{1} + \tau_r(T)\norm{x-\pi}{1} \notag\\
    \leq{}& (\tau_\ell(T) + \tau_r(T))\max(\norm{y-\pi}{1}, \norm{x-\pi}{1})\notag\\
    ={}& q\max(\norm{y-\pi}{1}, \norm{x-\pi}{1}),\label{eq:contraction}
  \end{align}
  where we have set $q = \tau_\ell(T) + \tau_r(T)$.
  We now prove inductively on $n\geq 0$, that for all $n\geq 0$, for all $i\geq 1$,
  \begin{equation}\label{eq:induction_hyp}
  \norm{\pi_i^{\sss (n)}-\pi}{1} \leq q^n \max_{j\geq 1} \norm{\pi_j^{\sss (0)} - \pi}{1}.
\end{equation}
  The base case holds straighforwardly; 
  for the induction, note that, for all $n\geq 1$, for all $i\geq 1$,
  because $\pi = \mathcal T(\pi, \pi)$ and~\eqref{eq:rec_laws}, we have
  \begin{align*}
    \norm{\pi_i^{\sss (n)}-\pi}{1} 
    ={}& \norm{\mathcal T\big(\pi_{2i}^{\sss (n-1)},\pi_{2i+1}^{\sss (n-1)}\big)-\mathcal T(\pi,\pi)}{1} \\
    \leq{}& q \max \big( \norm{\pi_{2i}^{\sss (n-1)}-\pi}{1}, \norm{\pi_{2i+1}^{\sss (n-1)}-\pi}{1} \big),
    \end{align*}
    by~\eqref{eq:contraction}.
    Now, by the induction hypothesis, we get
    \begin{align*}
    \norm{\pi_i^{\sss (n)}-\pi}{1}
    \leq{}& q \max \big( q^{n-1} \max_{j\geq 1} \norm{\pi_j^{\sss (0)}-\pi}{1}, q^{n-1}\max_{j\geq 1}\norm{\pi_j^{\sss (0)}-\pi}{1} \big) 
    \leq q^n \max_{j\geq 1}\norm{\pi_j^{\sss (0)}-\pi}{1},
  \end{align*}
  which concludes the proof of~\eqref{eq:induction_hyp}. The result follows by taking $i=1$ in~\eqref{eq:induction_hyp}.
\end{proof}


\subsection{The randomly labelled DAG is ``locally'' a 2-dependent Markov chain}
In this section, we consider the $\mathcal R$-labelled DAG as defined in Section~\ref{sub:coupling}.
First recall that the shape of the $\mathcal R$-labelled DAG is that of a uniform recursive DAG, 
As in Section~\ref{sec:shape-gn}, 
$H_n$ is the set of all the ancestors $v$ of node~$n$ such that $v \geq n_1$.
Also, $H'_n$ is the sub-graph of $H_n$ in which 
we only keep the nodes that are at graph distance at most $\ell(n)$ of node~$n$.
On $\mathcal E_n\cap \mathcal F_n$ (see~\eqref{eq:def_EF}), which occurs with high probability when $n$ is large (see Lemma~\ref{lem:EF}), $H'_n$ is a complete binary tree.

For all $k\geq 0$, 
we let $L(k)$ be the set of all nodes in $H_n$ at distance $k$ of node $n$. 
(Note that $L(k)$ is empty for all $k$ large enough.)

\begin{lemma}\label{lem:intralayer-indep}
Conditionally on $G_{n_1(n)}$ and $(X(i))_{0\leq i\leq n_1(n)}$,
and on the event $\mathcal E_n\cap \mathcal F_n$, for all $1\leq k\leq \ell(n)-1$,
$(X(i))_{i\in L(k)}$ is a sequence of independent random variables.
\end{lemma}

\begin{proof}
Let $u, v\in L(k)$. On $\mathcal E_n\cap\mathcal F_n$, $H_u$ and $H_v$ are two disjoint sub-trees of $H_n$. The label of $u$ is sampled by applying the following procedure:
\begin{itemize}
\item First, complete the tree by adding leaves so that every node in $H_u$ has degree~3.
\item By definition of the $\mathcal R$-labelled DAG, these leaves are merged with nodes chosen uniformly and independently at random in $G_{n_1(n)}$.
\item Now, given the labels of the leaves of $H_u$, we sample the labels of all nodes in $H_u$ recursively using the fact that, if a node's parents are labelled by the couples $x = (x_1, x_2)$ and $y = (y_1, y_2)$, 
then its label is distributed as $\rho(x)\otimes\rho(y)$, and independent of the rest.
\end{itemize}
The label of $v$ is sampled similarly, and because $H_u$ and $H_v$ do not intersect, this can be done independently for each of the two trees, implying that $X(u)$ and $X(v)$ are independent, as claimed.
\end{proof}

\begin{lemma}\label{lem:link_MC}
Fix $n\geq 0$; we work conditionally on $G_{n_1(n)}$ and $(X(i))_{1\leq i\leq n_1(n)}$, and on the event $\mathcal E_n\cap\mathcal F_n$.
We order the nodes in $L(\ell(n)-1)$ from left to right so that $H'_n$ is a plane binary tree.
For all $1\leq i\leq 2^{\ell(n)-1}$, we let $\pi_i^{\sss (0)}$ be the distribution of the label of the $i$-th node in level $L(\ell(n)-1)$ (from left to right); if $i>2^{\ell(n)-1}$, we let $\pi_i^{\sss (0)}$ be a fixed, and arbitrary distribution on $\{1, \ldots, d\}^2$.

Then, in distribution, $X(n) = W_1^{\sss (n)}$, where $(W_i^{\sss (n)})_{n\geq 0, i\geq 1}$ is the 2-dependent Markov chain on $\{1, \ldots, d\}^2$, of transition tensor $T$ given by
\begin{equation}\label{eq:def_T}
T_{(i_1,i_2)(j_1,j_2)(k_1,k_2)} = R_{i_1 j_1 j_2} R_{i_2 k_1 k_2}/\sigma^2,
\end{equation}
for all $i_1,i_2,j_1,j_2,k_1,k_2 \in \{1, \ldots, d\}$.
\end{lemma}

\begin{proof}
This is straightforward from Lemma~\ref{lem:intralayer-indep} (which gives independence of the labels of the leaves - needed in the definition of a 2-dependent Markov chain), and from the definition of the $\mathcal R$-labelled DAG (which gives the replacement tensor, see Section~\ref{sub:coupling}).
\end{proof}

\begin{lemma}\label{lem:EimpliesT}
If $R = (R_{ijk})_{1\leq i,j,k\leq d}$ satisfies Assumption {\bf (E)}, then
$T$ satisfies Assumption {\rm (i)} of Lemma~\ref{lem:AssE}.
\end{lemma}

\begin{proof}
Let $i,j,j',k,k' \in \{1,\ldots, d\}^2 = \mathcal{S}$. In what follows, we write $i = (i_1, i_2)$ (and similarly for $j, j', k, k'$), where $i_1, i_2 \in \{1, \ldots, d\}$.
We have
\begin{align*}
\max_{j,k,k'\in\mathcal S} \sum_{i\in\mathcal S} \abs*{T_{ijk} - T_{ijk'}} 
&=\frac1{\sigma^2}\max_{j,k,k'\in\mathcal S} \sum_{i\in\mathcal S} \abs*{R_{i_1 j_1 j_2} R_{i_2 k_1 k_2} - R_{i_1 j_1 j_2} R_{i_2 k'_1 k'_2}}\\
&=\frac1{\sigma^2}\max_{j,k,k'\in\mathcal S} \sum_{i\in\mathcal S} \abs*{R_{i_1 j_1 j_2} (R_{i_2 k_1 k_2} - R_{i_2 k'_1 k'_2})}\\
&=\frac1{\sigma^2}\max_{j,k,k'\in\mathcal S} \sum_{i_2=1}^d \abs*{R_{i_2 k_1 k_2} - R_{i_2 k'_1 k'_2}}\sum_{i_1=1}^d R_{i_1 j_1 j_2}\\
&= \frac1{\sigma}\max_{k,k'\in\mathcal S} \sum_{i_2=1}^d \abs*{R_{i_2 k_1 k_2} - R_{i_2 k'_1 k'_2}},
\end{align*}
where we have used Assumption {\bf (B)}.
A similar calculation gives
\[\max_{j,j',k\in\mathcal S} \sum_{i\in\mathcal S} \abs*{T_{ijk} - T_{ij'k}}
\leq  \frac1{\sigma}\max_{j,j'\in\mathcal S} \sum_{i_1=1}^d \abs*{R_{i_1 j_1 j_2} - R_{i_1 j'_1 j'_2}}.\]
In total, we thus get that
\[\max_{j,k,k'\in\mathcal S} \sum_{i\in\mathcal S} \abs*{T_{ijk} - T_{ijk'}} + \max_{j,j',k\in\mathcal S} \sum_{i\in\mathcal S} \abs*{T_{ijk} - T_{ij'k}}
\leq \frac2{\sigma} \max_{j,j'\in\mathcal S} \sum_{i=1}^d \abs*{R_{i j_1 j_2} - R_{i j'_1 j'_2}} 
< 2,\]
by Assumption {\bf (E)}. This concludes the proof.
\end{proof}

\subsection{Proof of Theorem~\ref{thm:convergence}}\label{sub:proof}
Together with Proposition~\ref{lem:tensor-props}, Lemma~\ref{lem:EimpliesT} implies that
there exists a unique $\pi\in \Delta(\{1, \ldots d\}^2)$ such that $\mathcal T(\pi, \pi) = \pi$, 
and $W_1^{\sss (n)}$ converges in distribution to $\pi$.
We thus get, by Lemma~\ref{lem:link_MC}, that $X(n)\Rightarrow \pi$ in distribution as $n\uparrow\infty$, conditionally on $G_{n_1(n)}$ and $(X(i))_{1\leq i\leq n_1(n)}$, and on the event $\mathcal E_n\cap\mathcal F_n$.
This means that, for all $1\leq i,j\leq d$,
\[\frac1{\mathbb P(\mathcal E_n\cap \mathcal F_n)}\cdot\mathbb P\big(\{X(n) = (i,j)\}\cap \mathcal E_n\cap \mathcal F_n \mid G_{n_1(n)}, (X(i))_{1\leq i\leq n_1(n)}\big) \to \pi_{i,j},\]
as $n\uparrow\infty$. (This is because the event $\mathcal E_n\cap\mathcal F_n$ is independent of $G_{n_1(n)}$ and $(X(i))_{1\leq i\leq n_1(n)}$, by definition.)
Now, by Lemma~\ref{lem:EF}, $\mathbb P(\mathcal E_n\cap \mathcal F_n) \to 1$ as $n\uparrow\infty$.
Hence,
\[\mathbb P\big(\{X(n) = (i,j)\}\cap \mathcal E_n\cap \mathcal F_n \mid G_{n_1(n)}, (X(i))_{1\leq i\leq n_1(n)}\big) \to \pi_{i,j},\]
which, by the tower rule, implies
\[\mathbb P\big(\{X(n) = (i,j)\}\cap \mathcal E_n\cap \mathcal F_n\big) \to \pi_{i,j}.\]
Using again the fact that $\mathbb P(\mathcal E_n\cap \mathcal F_n) \to 1$ as $n\uparrow\infty$, we get
\[\mathbb P\big(X(n) = (i,j)\big) \to\pi_{i,j}.\]
Hence, to prove Theorem~\ref{thm:convergence}, it thus only remains to prove that $\pi = \nu\otimes \nu$, where $\nu$ is the unique solution
of $\sigma\nu = \mathcal R(\nu, \nu)$ in $\Delta_d$, 
where we interpret $\nu$ as a $d$-dimensional vector $\nu = (\nu(\{1\}), \ldots, \nu(\{d\}))$. 
We know that this solution exists and is unique by Assumption {\bf (E)}: 
Indeed, Assumption {\bf (E)} implies that the 3-mode tensor $R/\sigma$ satisfies condition (i) of Lemma~\ref{lem:AssE}, which, by Proposition~\ref{lem:tensor-props}, implies that $\nu = \mathcal R(\nu, \nu)/\sigma$ admits a unique solution in $\Sigma_d$.

To prove that $\pi = \nu\otimes \nu$, it is enough to prove that $\pi' := \nu\otimes \nu$ 
satisfies $\pi' = \mathcal T(\pi', \pi')$ (because the solution of this equation is unique, by Proposition~\ref{lem:tensor-props}).
First recall that we interpret the measure $\nu$ as a vector $(\nu_i)_{1\leq i\leq d}$. Similarly, we interpret $\pi'$ as a vector whose coordinates are $\pi'_{(i,j)} = \nu_i\nu_j$, for all $1\leq i,j,\leq d$.
By definition of $\mathcal T$, for all $x = (x_1, x_2)\in \{1, \ldots, d\}^2$,
\begin{align*}
\mathcal T(\pi', \pi')_x 
&= \sum_{(y_1, y_2),(z_1,z_2)\in\{1, \ldots, d\}^2} T_{xyz} \pi'_{(y_1, y_2)} \pi'_{(z_1, z_2)}\\
&= \frac1{\sigma^2}\sum_{(y_1, y_2),(z_1,z_2)\in\{1, \ldots, d\}^2} R_{x_1y_1y_2}R_{x_2z_1z_2} 
\pi'_{(y_1, y_2)} \pi'_{(z_1, z_2)},
\end{align*}
by~\eqref{eq:def_T}.
We thus get, by definition of $\pi' = \nu\otimes\nu$ (i.e.\ $\pi'_{(i,j)} = \nu_i\nu_j$, for all $1\leq i,j,\leq d$),
\begin{align*}
\mathcal T(\pi', \pi')_x
&= \frac1{\sigma^2} \bigg(\sum_{1\leq y_1, y_2\leq d}R_{x_1y_1y_2}\nu_{y_1}\nu_{y_2}\bigg)
\bigg(\sum_{1\leq z_1, z_2\leq d}R_{x_2z_1z_2}\nu_{z_1}\nu_{z_2}\bigg)\\
&= \frac{\mathcal R(\nu, \nu)_{x_1}}\sigma\cdot \frac{\mathcal R(\nu, \nu)_{x_2}}\sigma
= \nu_{x_1}\nu_{x_2} = \pi'_x,
\end{align*}
because $\mathcal R(\nu, \nu) = \sigma\nu$, and by definition of $\pi' = \nu\otimes \nu$.
We thus get that $\pi' = \mathcal T(\pi', \pi')$, which, by uniqueness, implies that $\pi = \pi' = \nu\otimes\nu$, as desired. This concludes the proof of Theorem~\ref{thm:convergence}.

\subsection{Proof of Theorem~\ref{thm:main-result}}
We have proved in Section~\ref{sub:coupling} that Theorem~\ref{thm:convergence} implies Theorem~\ref{thm:main-result}$(ii)$. Part $(i)$ is proved in Section~\ref{sub:proof}.

\section{Drawing more than two balls at a time}\label{sec:multi}
So far, we have restricted ourselves to drawing two balls at every time step in the urn process. This was mainly to keep the proofs as simple as possible, but extending the results and the proofs to drawing $m$ balls at each time step is straightforward. In this section, we explain how to do so.

\begin{definition}\label{def:m-multi-urn}
Let $R = (R_{i_1\ldots i_{m+1}} = R(i_1, \ldots, i_{m+1}))_{1\leq i_1, \ldots, i_m\leq d}$ be a collection of non-negative integers.
We define the {m-drawing d-colour Pólya urn with replacement tensor R}
as a Markov process $(U(n) = (U_1(n), \ldots, U_d(n)))_{n\geq 0}$ on $\mathbb N^d$ that satisfies, for all $n\geq 0$,
\[U_i(n+1) = U_i(n) + R(i, C_1(n), \ldots, C_m(n)),\]
where, given $U(n)$, $(C_1(n), \ldots, C_m(n))$ is a vector of i.i.d.\ random variables such that, for all $1\leq k\leq d$, 
\[\mathbb P(C_1(n) = k) = \frac{U_k(n)}{\|U(n)\|_1}.\]
\end{definition}
Let $\mathcal S = \{1,\ldots,d\}^m$. 
We assume that the replacement tensor $R$ satisfies the following assumptions:
\begin{itemize}
\item[{\bf (T)}] For all $1\leq i_1, \ldots, i_{m+1}\leq d$, $R_{ i_1, \ldots, i_{m+1}} \geq 0$.
\item[{\bf (B)}] There exists $\sigma \in (0,\infty)$ such that, for all $1\leq i_1, \ldots, i_m\leq d$, $\sum_{i=1}^d R_{ii_1\ldots i_m} = \sigma$.
\item[{\bf (E)}] $\max_{j,j' \in S} \sum_{i=1}^d \abs*{R_{ij_1\ldots j_m} - R_{ij'_1\ldots j'_m}} < 2\sigma/m$.
\end{itemize}

We define $\mathcal R = \mathbb R^d \times\cdots\times \mathbb R^d\to\mathbb R^d$ as follows: for all $x = (x_1, \ldots, x_m)\in\mathbb R^d \times\cdots\times \mathbb R^d$,
\[\mathcal R_i(x) = \sum_{1\leq i_1, \ldots, i_m\leq d} R_{ii_1\ldots i_m}x_1\cdots x_m,\]
and $\mathcal R(x) = (\mathcal R_1(x), \ldots, \mathcal R_d(x))$.
The multi-linear map $\mathcal R$ is called an $(m+1)$-mode tensor.

\begin{theorem}\label{thm:m-main-result}
  Let $(U(n))_n$ be an $m$-drawing $d$-colour balanced Pólya urn with replacement tensor $R$. 
  If $R$ satisfies Assumptions {\bf (T)}, {\bf (B)}, and {\bf (E)}, then
there exists a unique $x^*\in\Delta_d$ satisfying $\sigma x^* = \mathcal R(x^*, \ldots, x^*)$ and, in probbaility as $n\uparrow\infty$,
  \[\frac{U(n)}{\sum_{i=1}^d U_i(n)} \to x^*.\]
\end{theorem}

The proof of Theorem~\ref{thm:m-main-result} is a straightforward extension of the proof of Theorem~\ref{thm:main-result}. We briefly explain how the proof can be extended to this case.

The construction of the labelled DAG extends naturally, with each node having $m$ parents instead of $2$.
The events $\mathcal E_n$ and $\mathcal F_n$ become
\begin{align*}
\mathcal E_n &= \{H_n\text{ is an $m$-ary tree}\},\\
\mathcal F_n &= \{H'_n \text{ is a complete $m$-ary tree}\}.
\end{align*}
It is straightforward to prove that Lemma~\ref{lem:EF} also holds in this context (the fact that $\mathbb P(\mathcal E_n) \to 1$ as $n
\uparrow \infty$ is proved in~\cite{Jans14}).
Thus, Lemma~\ref{lem:intralayer-indep} extends straightforwardly,
and we get, conditionally on $G_{n_1(n)}$ and $(X(i))_{1\leq i\leq n_1(n)}$, and on the event $\mathcal E_n\cap \mathcal F_n$, $X(n) = W_1^{\sss (n)}$ in distribution, 
where $(W_i^{\sss (n)})_{n\geq 0, i\geq 1}$ is an $m$-dependent Markov chain taking values in $\{1, \ldots, d\}^m$ (defined on the infinite $m$-ary tree, as in Definition~\ref{def:MC} for the binary case).
The replacement tensor of this $m$-dependent Markov chain is given by
\[T_{x a^{\sss (1)}\ldots a^{\sss (m)}}
= \frac1{\sigma^m} \prod_{i=1}^{m} R_{x_i a_1^{\sss (i)} \ldots a_m^{\sss (i)}},\]
for all $x, a^{\sss (1)},\ldots, a^{\sss (m)}\in \mathcal S = \{1, \ldots, d\}^m$.
One can check that, if $R$ satisfies {\bf (E)}, then 
\begin{align*}
&\max_{a^{\sss (1)}, b^{\sss (1)}, a^{\sss (2)}, \ldots, a^{\sss (m)}\in\mathcal S}
    \sum_{x\in \mathcal S}
    \abs*{T_{xa^{\sss (1)}a^{\sss (2)}\ldots a^{\sss (m)}} - T_{xb^{\sss (1)}a^{\sss (2)}\ldots a^{\sss (m)}}} \\
&\hspace{1cm}+ \cdots + 
\max_{a^{\sss (1)}, a^{\sss (2)}, \ldots, a^{\sss (m)}, b^{\sss (m)} \in \mathcal S} 
\sum_{x\in \mathcal S}
    \abs*{T_{xa^{\sss (1)}\ldots a^{\sss (m-1)} a^{\sss (m)}} - T_{xa^{\sss (1)}\ldots a^{\sss (m-1)} b^{\sss (m)}}} 
< 2.
\end{align*}
Similarly to Proposition~\ref{lem:tensor-props}, this implies that $W_1^{\sss (n)}$ converges in distribution to $\pi$, the unique vector in
$\Delta(\mathcal S)$ satisfying $\pi = \mathcal T(\pi, \ldots, \pi)$.
To get existence and uniqueness of this vector, one needs to extend the proof of~\cite[Thm 5.1, Cor 5.2]{FT20}; this can be done using~\cite[Thm 4.4]{FT20} and via similar arguments as the ones presented in the proof of~\cite[Thm 5.1]{FT20}.

Finally, by {\bf (E)}, there exists a unique vector $\nu$ in $\Delta_d$ satisfying $\nu = \mathcal R(\nu, \ldots, \nu)$. As in Section~\ref{sub:proof}, one can check that $\pi = \nu\otimes \cdots \otimes \nu$ ($m$ times), and conclude the proof similarly as for the binary case.

\section{Examples}\label{sec:examples}

To denote the replacement tensor of a $2$-colour, $2$-drawing P\'olya urn, we  write
\[R =
  \begin{pmatrix}[cc|cc]
    R_{111} & R_{211} & R_{121} & R_{221}\\
    R_{112} & R_{212} & R_{122} & R_{222}
  \end{pmatrix},
\]
where we recall that $R_{ijk}$ is the number of balls of colour $i$ added in the urn when the colours of the couple of balls picked is $(j,k)$.

\subsection{Simple examples}

In this first section, we check that our main result gives the expected result (or non result) for three simple cases.

\begin{example}
For our first example, we consider the following $2$-colour, $2$-drawing case:
\[R =
  \begin{pmatrix}[cc|cc]
    2 & 0 & 1 & 1\\
    1 & 1 & 0 & 2
  \end{pmatrix}
\]
i.e.\ $R_{ijk} = {\bf 1}_{i=j} + {\bf 1}_{i=k}$. In other words, if we pick a couple of balls of colour $(j,k)$, then we add one ball of colour $j$ and one ball of colour $k$ into the urn. Intuitively, this is a generalisation of the (classical) P\'olya urn with identity replacement matrix (at every time step, we add a ball of the same colour as the one picked at that step). For this classical P\'olya urn, it is well known that $U(n)$ converges almost surely to a random variable. So we do not expect Theorem~\ref{thm:main-result} to apply.
Indeed, Assumption {\bf (E)} is not satisfied as
  \[\max_{j,j'\in\mathcal S} \sum_{i=1}^d \abs*{R_{i j_1 j_2} - R_{i j'_1 j'_2}} = 4 > 2 = \sigma.\]
\end{example}

\begin{example}
We now consider the following $2$-colour, $2$-drawing case: 
\[R =
    \begin{pmatrix}[cc|cc]
      1 & 1 & 1 & 1\\
      1 & 1 & 1 & 1\\
    \end{pmatrix},
  \]
  i.e.\ $R_{ijk} = 1$ for all $1\leq i,j,k\leq 2$. In words, this means that, at every time step, we add $1$ ball of colour $1$ and one ball of colour $2$, independently from the draw. This case is a thus a bit trivial: one expects that $U(n)/n \to (\nicefrac12, \nicefrac12)$ almost surely as $n\uparrow\infty$.
Indeed, Theorem~\ref{thm:main-result} applies: Assumptions {\bf (T)} and {\bf (B)} hold trivially, and 
  \[\max_{j,j'\in\mathcal S} \sum_{i=1}^d \abs*{R_{i j_1 j_2} - R_{i j'_1 j'_2}} = 0 < 2 = \sigma,\]
so {\bf (E)} holds as well.
Now note that, by definition, 
$\mathcal R(x, y)_i = \sum_{j=1}^2\sum_{k=1}^2 x_j y_k = (x_1+x_2)(y_1+y_2)$, and
\[\mathcal R(\nu, \nu) = \nu\quad \Leftrightarrow\quad
\nu_1 = \nu_2 = (\nu_1+\nu_2)^2 \quad \Leftrightarrow\quad
\nu_1 = \nu_2 = 4\nu_1^2 \quad \Leftrightarrow\quad
\nu_1 = \nu_2 = \nicefrac12.
\]
So Theorem~\ref{thm:main-result} gives that $U(n)/n\to (\nicefrac12, \nicefrac12)$ {in probability} as $n\uparrow\infty$, as
expected.
\end{example}

\begin{example}
  Consider a matrix for a single-drawing Pólya urn
  \[A =
    \begin{pmatrix}
      a_{11} & a_{12} \\ a_{21} & a_{22}
    \end{pmatrix}.
  \]
  This can be transformed into a multi-drawing Pólya urn that only considers the first draw for its replacement as follows:
  \[R_{ij_1j_2\ldots j_m} = a_{ij_1}.\]
  The ergodicity coefficient for a single-drawing Pólya urn~\cite{FT20} is given by
  \[\max_{j,j' \in \{0,\ldots,d\}} \sum_{i=1}^d \abs*{a_{ij}-a_{ij'}}.\]
  We recover this coefficient when applying Assumption \textbf{(E)}:
  \[\max_{j,j' \in \mathcal S} \sum_{i=1}^d \abs*{R_{ij_1\ldots j_m} - R_{ij_1'\ldots j'_m}} = \max_{j_1,j_1' \in \{1,\ldots,d\}}
    \sum_{i=1}^d \abs*{a_{ij_1} - a_{ij_1'}}.\]
  However, the upper bound given by Assumption \textbf{(E)}, $2 \sigma/m$, is stronger than the upper bound for a single-drawing Pólya urn,
  which is just $2 \sigma$.
  A possible explanation for this gap is that, even though only the first draw matters for the distribution of the balls present in the urn,
  our result considers the distribution of all $m$ draws and their joint convergence.
\end{example}

\subsection{Affine urns and non-symmetric urns}\label{sub:affine}

\begin{example}
In~\cite{KM}, 
Kuba and Mahmoud prove convergence theorems for $2$-colour, 
$m$-drawing balanced ``affine'' P\'olya urns.
We focus here in the $2$-drawing case ($m=2$).
They say that such an urn is affine if 
\[R =
    \begin{pmatrix}[cc|cc]
      a_0 & \sigma-a_0 & a_1 & \sigma-a_1 \\
      a_1 & \sigma-a_1 & a_2 & \sigma-a_2
    \end{pmatrix},
  \]
where there exists an integer $h$ such that $a_i = a_0 + ih$, for all $i\in\{1,2\}$. 
To apply our results, we assume that $0\leq a_0, a_1, a_2 \leq \sigma$ (which is not required in~\cite{KM}), which implies $h < \sigma$.
Then both Assumptions {\bf (T)} and {\bf (B)} hold.
Now, Assumption {\bf (E)} is equivalent to
\[2\max\{|a_1-a_0|,|a_2-a_0|,|a_2-a_1|\} < \sigma
\quad\Leftrightarrow\quad
4h<\sigma.\]
Under this extra assumption, not required in~\cite{KM}, we get that
$U(n)/n \to x^*$ {in probability} as $n\uparrow\infty$, where $x^* = \mathcal R(x^*, x^*)/\sigma$.
In this case, for all $x\in \Delta_2$,
\[\mathcal R(x,x) = 
\begin{pmatrix}
a_0 x_1^2 + 2a_1 x_1x_2 + a_2 x_2^2\\
\sigma - \big(a_0 x_1^2 + 2a_1 x_1x_2 + a_2 x_2^2\big).
\end{pmatrix}
\]
One can check that the unique solution of $x^* = \mathcal R(x^*, x^*)/\sigma$ is given by
$x^* = \big(\frac{a_0 +2h}{\sigma+2h}, 1- \frac{a_0 +2h}{\sigma+2h}\big)$.
This recovers the result of~\cite[Proposition~5]{KM}, albeit under stronger assumptions, and with convergence in probability instead of almost-surely.
Note that, in the case of balanced affine urns, martingale theory can be applied to get stronger results 
(and it is thus not surprising that our result gives a weaker result than the results specific to affine urns in the literature): 
Kuba and Mahmoud~\cite{KM} get second order convergence results, Kuba and Sulzbach~\cite{KS17} prove a law of the iterated logarithm, and Sparks, Kuba, Balaji and Mahmoud~\cite{SKBM25} prove second order results for the $d$-colour case, $d\geq 2$.
\end{example}

A strength of our approach is that the order in which the balls were drawn from the urn may matter for the replacement tensor $R$, which has
not been considered in existing literature on multi-drawing Pólya urns.

\begin{example}
  Consider a two-colour, two-drawing Pólya urn with the following replacement tensor:
  \[R =
    \begin{pmatrix}[cc|cc]
      1 & 2 & 1 & 2\\
      2 & 1 & 1 & 2
    \end{pmatrix}.
  \]
  This urn is non-symmetric, as when we draw first a ball of type $1$, then a ball of type $2$, we replace $2$ balls of the first and one of
  the second type, whereas this is reversed if the order of the draws is reversed. The urn fulfills Assumptions \textbf{(T)} and
  \textbf{(B)}, with $\sigma = 3$.

  Assumption \textbf{(E)} is also satisfied since
  \[\max_{j,j'\in\mathcal S} \sum_{i=1}^d \abs*{R_{i j_1 j_2} - R_{i j'_1 j'_2}} = 2 < 3 = \sigma,\]
  so a unique solution $x^* \in \Delta_2$ for $\mathcal R(x,x)/\sigma = x$ exists that the proportion of balls in the urn will converge to
  in distribution. Solving the fixed-point equation $x^* = \mathcal R(x^*, x^*)/\sigma$ yields $x^* = (\sqrt{2}-1,2-\sqrt{2})$ as the unique
  solution in $\Delta_2$.
\end{example}

\begin{example}
  Consider a two-colour, two-drawing Pólya urn with the following replacement tensor:
  \[R = 
    \begin{pmatrix}[cc|cc]
      0 & 5 & 1 & 4\\
      2 & 3 & 2 & 3
    \end{pmatrix}.
  \]
  This urn is again non-symmetric and fulfills Assumptions \textbf{(T)} and
  \textbf{(B)}, with $\sigma = 5$.

  Assumption \textbf{(E)} is also satisfied since
  \[\max_{j,j'\in\mathcal S} \sum_{i=1}^d \abs*{R_{i j_1 j_2} - R_{i j'_1 j'_2}} = 4 < 5 = \sigma,\]
  so a unique solution $x^* \in \Delta_2$ for $\mathcal R(x,x)/\sigma = x$ exists that the proportion of balls in the urn will converge to
  in distribution. The unique solution to the fixed-point equation $x^* = \mathcal R(x^*, x^*)/\sigma$ is given by $x^* = (\sqrt{11}-3,4-\sqrt{11})$.
\end{example}

\subsection{Examples from~\cite{LMS18}}\label{sub:tunisian}

In~\cite{LMS18}, Lasmar, Mailler and Selmi consider the following multi-drawing model: At each timestep $n$, draw $m$ balls from the urn and
denote by $\xi_{n,i}$ the number of balls of colour $i$ drawn. If the urn has $d$ colours, $\xi_n$ is a $d$-dimensional vector with $\sum_i
\xi_{n,i} = m$. Then, $\hat R(\xi_n)_i$ balls of type $i$ are replaced into the urn for each colour $i$. Again, the urn is called balanced, if for
all such $\xi_n$, $\sum_i \hat R(\xi_{n,i}) = \sigma$. In fact, this model is subsumed in our approach, as we can transfer this replacement
function $\hat R$ into a tensor $(R_{ij_1\ldots j_m})_{ij_1\ldots j_m = 1}^d$ as follows:
\[R_{ij_1\ldots j_m} = \hat R(v)_i, \quad \text{with} \quad v_k = \sum_{\ell=1}^m \1_{j_\ell = k}.\]
The set of all such vectors $v$ is also the set of all possible draws in the model of Lasmar, Mailler and Selmi and defined as
\[\Sigma_m^{(d)} \coloneqq \Set{x \in \N^d \given \sum_{i=1}^d x_i = m}.\]

In the analysis of Lasmar, Mailler and Selmi~\cite{LMS18}, the set of zeros of the following function
\begin{align*}
  & h: \Set{x \in [0,1]^d \given \sum_{i=1}^d x_i = 1} \to \Set{x \in \R^d \given \sum_{i=1}^d x_i = 0} \\
  & h(x) = \sum_{v \in \Sigma_m^{(d)}} \binom{m}{v_1,\ldots,v_d} (\prod_{k=1}^d x_k^{v_k}) (\hat R(v) - \sigma x) \addtag \label{eq:h-lms18}
\end{align*}
plays a central role. We consider the $i$th component of $h(x)$ and rewrite for $x \in \Set{x \in [0,1]^d \given \sum_i x_i = 1}$,
\begin{align*}
  h_i(x) \coloneqq{}& \sum_{v \in \Sigma_m^{(d)}} \binom{m}{v_1,\ldots,v_d} (\prod_{k=1}^d x_k^{v_k}) (\hat R_i(v) - \sigma x_i) \\
  ={}& \sum_{v \in \Sigma_m^{(d)}} \binom{m}{v_1,\ldots,v_d} (\prod_{k=1}^d x_k^{v_k}) \hat R_i(v)
       - \sigma x_i \underbrace{\sum_{v \in \Sigma_m^{(d)}} \binom{m}{v_1,\ldots,v_d} (\prod_{k=1}^d x_k^{v_k})}_{=1 \text{ (multinom. distr.)}}\\
  ={}& \sum_{v \in \Sigma_m^{(d)}} \binom{m}{v_1,\ldots,v_d} (\prod_{k=1}^d x_k^{v_k}) \hat R_i(v) - \sigma x_i.
\end{align*}

Our analysis requires the study of
\[\sigma x = \mathcal R(x,\ldots,x) \iff \mathcal R(x,\ldots,x) - \sigma x \equiv 0, \addtag \label{eq:R-fp}\]
for $x \in \Set{x \in [0,1]^d \given \sum_{i=1}^d x_i = 1}$. Expanding the tensor multiplication for the $i$th component leads us to
\begin{align*}
  \mathcal R(x,\ldots,x)_i -\sigma x ={}& \sum_{j_1,\ldots,j_m=1}^d R_{ij_1,\ldots j_m}x_{j_1}\ldots x_{j_m} - \sigma x \\
  \intertext{Via the earlier definition of $R$, the counts recorded in $v$ with $v_k = \sum_{\ell=1}^m \1_{j_\ell = k}$ uniquely determine the value of $R_{ij_1\ldots j_m}$. The number of possible index choices to achieve a count vector $v$ is given by the multinomial coefficient.}
  ={}& \sum_{v \in \Sigma_m^{(d)}} \binom{m}{v_1,\ldots,v_m} \hat R(v)_i x_1^{v_1} \ldots x_d^{v_d}  - \sigma x \\
  ={}& \sum_{v \in \Sigma_m^{(d)}} \binom{m}{v_1,\ldots,v_m} \hat R(v)_i \prod_{k=1}^d x_k^{v_k} - \sigma x.
\end{align*}

We see that Equations~\eqref{eq:h-lms18} and \eqref{eq:R-fp} reduce to the same expression for $x \in \Set{x \in \R^d \given \sum_{i=1}^d x_i = 1}$ and therefore have
the same solution set in this space. We highlight this by looking at some of the examples presented in~\cite{LMS18}.

\begin{example}
Lasmar, Mailler and Selmi~\cite{LMS18} consider the 2-drawing P\'olya urn with replacement tensor
  \[R =
    \begin{pmatrix}[cc|cc]
      1 & 2 & 2 & 1\\
      2 & 1 & 1 & 2
    \end{pmatrix}.
  \]
  This is a non-affine urn with $\sigma = 3$, fulfilling Assumptions \textbf{(T)} and \textbf{(B)}. 
 Assumption \textbf{(E)} is also satisfied since
  \[\max_{j,j'\in\mathcal S} \sum_{i=1}^d \abs*{R_{i j_1 j_2} - R_{i j'_1 j'_2}} = 2 < 3 = \sigma,\]
  so a unique solution $x^* \in \Delta_d$ for $\mathcal R(x,x)/\sigma = x$ exists. 
  Solving the resulting system of equations
\[\mathcal R(x,x)/\sigma =
\begin{pmatrix} \frac13(x_1^2 + 4x_1x_2 + x_2^2)\\
\frac23(x_1^2 + x_1x_2 + x_2^2)
\end{pmatrix} =
\begin{pmatrix}
  x_1 \\ x_2
\end{pmatrix}
\]
yields $x_1 = x_2 = \nicefrac12$.
Theorem~\ref{thm:main-result} thus implies that the proportion of balls of each of the two colours converges in probability to $\nicefrac12$, recovering partially the result from~\cite{LMS18} (only partially because the convergence is almost sure in~\cite{LMS18}).
\end{example}

While~\cite[Examples 2-5]{LMS18} also all fulfill Assumptions \textbf{(T)} and \textbf{(B)}, none of them fulfill Assumption \textbf{(E)} and as such we cannot treat them with our result.
This is again not surprising, as the theory of stochastic approximations can be used to
control the entire evolution of a process, while our result only looks at the final phase of the process. We still give an exemplary
treatment of Examples 2 and 3 in the following.

\begin{example} Consider the 2-drawing urn of replacement tensor 
  \[
    R =
    \begin{pmatrix}[cc|cc]
      4 & 0 & 1 & 3\\
      1 & 3 & 1 & 3
    \end{pmatrix}.
  \]
  It satisfies assumptions {\bf (B)} with $\sigma = 4$ and {\bf (T)}.
  However,
  \[\max_{j,j'\in\mathcal S} \sum_{i=1}^d \abs*{R_{i j_1 j_2} - R_{i j'_1 j'_2}} = 6 > 4 = \sigma,\]
  i.e.\ Assumption {\bf (E)} is not satisfied.
Note that the equation $\mathcal R(x,x) = 4 x$ admits two solutions in $\Delta_2$: $(\nicefrac13, \nicefrac23)$ 
and $(1,0)$. 
Lasmar, Mailler and Selmi~\cite{LMS18} are able to prove that, almost surely as $n\uparrow\infty$, $U(n)/n \to x^* = (\nicefrac12,\nicefrac23)$. Our result does not apply to this example.
\end{example}

\begin{example}Another example in Lasmar, Mailler and Selmi~\cite{LMS18} is the 2-drawing urn of replacement tensor
  \[
    R =
    \begin{pmatrix}[cc|cc]
      7 & 1 & 3 & 5\\
      3 & 5 & 1 & 7
    \end{pmatrix}.
  \]
  It satisfies Assumptions {\bf (T)} and {\bf (B)} with $\sigma = 8$.
  However, Assumption \textbf{(E)} is not satisfied since
  \[\max_{j,j'\in\mathcal S} \sum_{i=1}^d \abs*{R_{i j_1 j_2} - R_{i j'_1 j'_2}} = 12 > 8 = \sigma,\]
  Note that $\mathcal R(x,x) = 8x$ admits a unique solution in $\Delta_2$: $(1-\nicefrac1{\sqrt{2}},\nicefrac1{\sqrt{2}})$.
  Lasmar, Mailler and Selmi~\cite{LMS18} are able to prove that, almost surely as $n\uparrow\infty$, $U(n)/n \to x^* = (1-\nicefrac1{\sqrt{2}},\nicefrac1{\sqrt{2}})$. Our result does not apply to this example.
\end{example}

\subsection{Further examples from the literature}

As we have seen, the presence of a unique solution to $\mathcal R(x,x) = \sigma x$ is not sufficient for Assumption {\bf (E)} to hold and thus for our result to apply. 
This also holds when
considering higher-order Markov chains with a simpler dependency structure, 
leading to the analysis of power iteration schemes such as
$x_{t+1} =\mathcal R(x_t, x_t)$. 
This is illustrated by the following 
example.

\begin{example}[\protect{\cite[Ex. 3]{LN14}}]
  The urn given by
  \[R =
    \begin{pmatrix}
      0 & 1 & 0 & 1 \\
      1 & 0 & 1 & 0
    \end{pmatrix}
  \]
  satisfies Assumptions \textbf{(T)} and \textbf{(B)} with $\sigma=1$ and admits a unique solution to $\mathcal R(x,x) = x$ with $x =\frac12(1,1)$.
However, Assumption \textbf{(E)} is not satisfied since
  \[\max_{j,j'\in\mathcal S} \sum_{i=1}^d \abs*{R_{i j_1 j_2} - R_{i j'_1 j'_2}} = 2 > 1 = \sigma.\]
 Li and Ng~\cite{LN14} prove that a simple power iteration scheme already does not converge to this fixed point.
\end{example}

Further, for a stochastic matrix $P$, positivity is sufficient to imply contractivity of the map $x \mapsto Px$ and the existence of a
unique fixed point $x^*$. This is not the case with stochastic tensors, as illustrated by the following example.

\begin{example}
  A 3-draw, 2-colour case (see~\cite[Example 1.7]{CZ13}): 
  \begin{alignat*}{6}
    R_{1111}&{}= 0.872 &\qquad& R_{1112}&{}= 2.416/3 &\qquad& R_{1121}&{}=2.416/3 &\qquad& R_{1122}&{}=0.616/3\\
    R_{1211}&{}= 2.416/3 &\qquad& R_{1212}&{}= 0.616/3 &\qquad& R_{1221}&{}=0.616/3 &\qquad& R_{1222}&{}=0.072\\
    R_{2111}&{}= 0.128 &\qquad& R_{2112}&{}= 0.584/3 &\qquad& R_{2121}&{}=0.584/3 &\qquad& R_{2122}&{}=2.384/3\\
    R_{2211}&{}= 0.584/3 &\qquad& R_{2212}&{}= 2.384/3 &\qquad& R_{2221}&{}=2.384/3 &\qquad& R_{2222}&{}=0.928,
  \end{alignat*}
  This tensor has strictly positive entries and satisfies Assumption \textbf{(B)} with $\sigma=1$, however, $\mathcal R(x,x,x) = x$ is
  solved by both $v_1 = (0.2,0.8)$ and
  $v_2 = (0.6,0.4)$. When we check Assumption \textbf{(E)} by comparing all pairs $(R_{1j_1j_2j_3},
  R_{2j_1j_2j_3})$ for maximal absolute value difference, we find
  \[\max_{j,j'\in\mathcal S} \sum_{i=1}^d \abs*{R_{i j_1 j_2 j_3} - R_{i j'_1 j'_2 j'_3}} = 1.6 > 1.\]
\end{example}


\section{Infinitely-many colour case}\label{sec:infinite}
As mentioned in the introduction, one can easily generalise our model to having infinitely-many colours. In this section, we define this more general model, and state and prove an equivalent of Theorem~\ref{thm:main-result} for this case.

\subsection{Definition of infinitely-many colour, multi-drawing urns}
Let $\mathcal C$ be a measurable space.
Let $(\rho_{x,y})_{x,y\in\mathcal C}$ be a kernel of non-negative measures on $\mathcal C$ such that, for 
all $x,y\in\mathcal C$, $\rho_{x,y}(\mathcal C) = \sigma>0$.

We define a process of random non-negative measures on $\mathcal C$ recursively as follows:
Let $m_0$ be a finite non-negative measure on $\mathcal C$ and, for all $n\geq 0$, given $m_n$, 
we sample $(C_1(n+1), C_2(n+1))$ a random variable of distribution $\hat m_n\otimes \hat m_n$, where $\hat m_n = m_n/m_n(\mathcal C)$, and set
\[m_{n+1} = m_n + \rho_{C_1(n+1), C_2(n+1)}.\]
The process $(m_n)_{n\geq 0}$ is the two-drawing P\'olya urn of initial composition $m_0$ and replacement kernel $(\rho_{x,y})_{x,y\in\mathcal C}$.

\begin{remark}
This model is indeed an extension of the finitely-many colour case: in Section~\ref{sub:MVPP}, we have already showed how to see the finitely-many colour case as a measure-valued process.
\end{remark}

\subsection{Two-dependent Markov chains on the infinite binary tree}
The theory of two-dependent Markov chains introduced in Section~\ref{sec:two-dep-mc} can be extended to Markov chains taking values on an infinite set $\mathcal S$.
The definition is as follows:
\begin{definition}\label{def:MC_inf}
Let $(\pi_i)_{i\geq 1}$ be a sequence of probability distributions on a finite or countable space $\mathcal S$.
Let $(\tau_{x,y})_{x,y \in \mathcal S}$ 
be a kernel of non-negative measures on $\mathcal S$ satisfying $\tau_{x,y}(\mathcal S)=1$ for all $x,y\in\mathcal S$.
We sample $W^{\sss (0)}_i\sim \pi_i$ for all $i\geq 1$, independently of each other.
We then sample the values of $(W_i^{\sss (n)})_{i\geq 1}$ for all $n\geq 1$ recursively as follows: 
for all $n\geq 1$, for all $i\geq 1$, for all $x, y,\in \mathcal S$, given $(W^{\sss (n-1)}_{2i}, W^{\sss (n-1)}_{2i+1}) = (x,y)$, sample $W_i^{\sss (n)}$ a random variable of distribution $\tau_{x,y}$.
We call this process $(W_i^{\sss (n)})_{i\geq 1, n\geq 0}$ the 2-dependent Markov chain of transition tensor $(\tau_{x,y})_{x,y \in \mathcal S}$ and initial distribution $(\pi_i)_{i\geq 1}$.
\end{definition}

Similarly to the finitely-many colour case, 
a 2-dependent Markov chain can be interpreted as a random labelling of the infinite binary tree.
The value $W_i^{\sss (n)}$ can be interpreted as the label of node $(n,i)$: its distribution given the labels of its two parents is $\tau_{(W_{2i}^{\sss (n-1)}, W_{2i+1}^{\sss (n-1)})}$.

In the following, we let $\mathcal M(\mathcal S)$ be the set of all signed measures on $\mathcal S$, and $\mathcal P(\mathcal S)$ be the set of all probability measures on $\mathcal S$.

We define the operator $\mathcal T = \mathcal M(\mathcal S) \times \mathcal M(\mathcal S) \to \mathcal M(\mathcal S)$ as follows: for any two signed measures $\mu$ and $\nu$ on $\mathcal S$
\[\mathcal T(\mu, \nu)
= \int_{x\in\mathcal S}\int_{y\in\mathcal S} 
\tau_{x,y} \mathrm d\mu(x)\mathrm d\nu(y).\]

\begin{lemma}\label{lem:xin_inf}
Let $(\pi_i)_{i\geq 1}$ be a sequence of probability distributions on a finite or countable space $\mathcal S$.
Let $(\tau_{x,y})_{x,y \in \mathcal S}$ 
be a kernel of non-negative measures on $\mathcal S$ satisfying $\tau_{x,y}(\mathcal S)=1$ for all $x,y\in\mathcal S$.
Let $(W_i^{\sss (n)})_{i\geq 1, n\geq 0}$ be the 2-dependent Markov chain of transition tensor $(\tau_{x,y})_{x,y \in \mathcal S}$ and initial distribution $(\pi_i)_{i\geq 1}$. 
For all $n\geq 0, i\geq 1$, we let $\pi_i^{\sss (n)}$ be the distribution of $W_i^{\sss (n)}$.
Then,  for all $i\geq 1$, $\pi_i^{\sss (0)} = \pi_i$ and, for all $n\geq 1$, $i\geq 1$, $w\in\mathcal S$,
\begin{equation}\label{eq:rec_laws_inf}
\pi_i^{\sss (n)} = \mathcal T(\pi_{2i}^{\sss (n-1)}, \pi_{2i+1}^{\sss (n-1)}).
\end{equation}
\end{lemma}

\begin{definition}\label{def:inv_distr_inf}
We say that $\pi \in\mathcal P(\mathcal S)$ is an invariant distribution of the operator $\mathcal T$ 
if and only if $\pi = \mathcal T(\pi, \pi)$.
\end{definition}

Similarly to the finitely-many colour case, we define the following two maps: 
for all $\mu,\nu\in\mathcal M(\mathcal S)$, we let
\begin{equation}\label{eq:def_Rlr_inf}
\mathcal T^{(\ell)}_\mu (\nu) 
= \mathcal T(\mu, \nu)
\quad\text{ and }\quad 
\mathcal T^{(r)}_\nu (\mu) 
= \mathcal T(\mu, \nu).
\end{equation}
If $\tau_{x,y}(\mathcal S)=1$ for all $x,y\in\mathcal S$, 
then $\mathcal T^{(\ell)}_\mu$ and $\mathcal T^{(r)}_\mu$ are Markovian kernels on $\mathcal S$, 
for all $\mu\in\mathcal P(\mathcal S)$. 
We can define the norm-based ergodicity coefficient (sometimes called Doeblin or Dobrushin coefficient) of a Markovian kernel $\mathcal K$ on
$\mathcal S$ as follows:
\[ \vvvert\mathcal K \vvvert
= \sup_{\nu_1, \nu_2\in \mathcal P(\mathcal S)} \frac{\|\mathcal K(\nu_1) - \mathcal K(\nu_2)\|_{\mathrm{TV}}}{\|\nu_1-\nu_2\|_{\mathrm{TV}}}.\]
We let
  \begin{equation}\label{eq:def_taus_inf}
  \tau_\ell(\mathcal T) \coloneqq{} \sup_{\mu\in\mathcal P(\mathcal S)} \vvvert\mathcal T^{(\ell)}_\mu\vvvert
    \quad\text{ and }\quad 
     \tau_r(\mathcal T) \coloneqq{} \sup_{\mu \in \mathcal P(\mathcal S)} \vvvert\mathcal T^{(r)}_\mu\vvvert.
  \end{equation}

\subsection{Main result for infinitely-many colours}
As in the finitely-many colour case, we assume that the urn process is balanced:
\begin{itemize}
\item[{\bf (B)}] There exists $\sigma \in (0,\infty)$ such that, for all $x,y\in\mathcal C$, 
$\rho_{x,y}(\mathcal C) = \sigma$.
\end{itemize}
Tenability is already ensured by the fact that, for all $x,y\in\mathcal C$, $\rho_{x,y}$ is a non-negative measure.
The ergodicity assumption however, is much harder to state, mostly because the results of~\cite{FT20} are not available in this setting. To
state it, we consider the operator $\mathcal T$ for the 2-dependent Markov chain on $\mathcal S = \mathcal C \times \mathcal C$: For all
$\mu, \nu\in\mathcal M(\mathcal C \times \mathcal C)\times\mathcal M(\mathcal C \times \mathcal C)$,
\begin{equation}\label{eq:def_T_inf}
\mathcal T(\mu, \nu) = \frac1{\sigma^2} \int_{(x_1, x_2)\in\mathcal C\times\mathcal C} \int_{(y_1, y_2)\in\mathcal C\times\mathcal C}
\rho_{x_1, x_2} \otimes \rho_{y_1, y_2} \mathrm d\mu(x_1, x_2)\mathrm d\nu(y_1, y_2).
\end{equation}
\begin{itemize}
\item[{\bf (E)}] $\tau_\ell(\mathcal T) + \tau_r(\mathcal T)<1$.
\end{itemize}
Analogously to the finitely-many colour case, we also define $\mathcal R = \mathcal M(\mathcal C)\times \mathcal M(\mathcal C)\to \mathcal M(\mathcal C)$ as follows: for all $\mu, \nu\in \mathcal M(\mathcal C)$,
\begin{equation}\label{eq:def_R_inf}
\mathcal R(\mu, \nu) = \int_{x\in\mathcal C}\int_{y\in\mathcal C}\rho_{x,y}\mathrm d\nu(x)\mathrm d\mu(y).
\end{equation}

\begin{theorem}\label{thm:convergence_inf}
Let $(m_n)_{n\geq 0}$ be a two-drawing P\'olya process on a measurable space $\mathcal C$, with initial composition $m_0$ and replacement kernel $(\rho_{x,y})_{x,y\in\mathcal C}$. If $(\rho_{x,y})_{x,y\in\mathcal C}$ satisfies Assumptions {\bf (B)} and {\bf (E)}, then there exists a unique $\nu\in\mathcal P(\mathcal C)$ satisfying 
\begin{equation}\label{eq:fix_point_R_inf}
\sigma\nu =\mathcal R(\nu, \nu).
\end{equation}
and, in probability as $n\uparrow\infty$,
$\hat m_n\to \nu$.
\end{theorem}

\begin{remark}
Because of Assumption {\bf (E)} being hard to check in practice, we are not able to provide any (non-trivial) example of infinitely-many colour multi-drawing P\'olya urn to which Theorem~\ref{thm:convergence_inf}.
We leave it for further work to prove an equivalent of Fasino and Tudisco~\cite{FT20}'s criterion for $2$-dependent Markov chains on infinite state spaces.
\end{remark}

\subsection{Proof of Theorem~\ref{thm:convergence_inf}}

\begin{proposition}\label{lem:tensor-props_inf}
  Let $\mathcal T$ be a 2-dependent Markov operator on $\mathcal S$ satisfying condition \textbf{(E)}. Then there exists a unique $\pi \in
  \mathcal P (S)$ such that $\mathcal T(\pi, \pi) = \pi$. 
Furthermore, 
any sequence $(\pi_i^{\sss (n)})_{n\geq 0, i\geq 1}$ of measures in $\mathcal P(\mathcal S)$ that satisfies~\eqref{eq:rec_laws_inf} is such that, for all $n\geq 0$,
\[\norm{\pi_1^{\sss (n)}-\pi}{1} \leq q^n \max_{i\geq 1} \norm{\pi_i^{\sss (0)} - \pi}{1},\]
for $q = \tau_\ell(\mathcal T) + \tau_r(\mathcal T)<1$.
Therefore, $\lim_{n \to \infty} \pi_1^{\sss (n)} = \pi$.
\end{proposition}

Note that, by Lemma~\ref{lem:xin_inf}, $\lim_{n \to \infty} \pi_1^{\sss (n)} = \pi$ implies that $W_1^{\sss (n)} \Rightarrow \pi$ in
distribution as $n\uparrow\infty$, where $(W_i^{\sss (n)})_{n\geq 0, i\geq 1}$ is the 2-dependent Markov chain with operator $\mathcal T$.

\begin{proof}
  Let $\tilde{\mathcal T}(\pi) = \mathcal T(\pi, \pi)$. Then $\tilde{\mathcal T}$ is a map from $\mathcal P(\mathcal S)$ to $\mathcal P(\mathcal S)$. 
By linearity of the integral, for all $\mu, \nu\in\mathcal M(\mathcal S)$,
  \begin{align*}
    \tilde{\mathcal T}(\mu) - \tilde{\mathcal T}(\nu) ={}& \mathcal T(\mu,\mu) - \mathcal T(\nu,\nu) \\
    ={}& \mathcal T(\mu,\mu) - \mathcal T(\mu,\nu) + \mathcal T(\mu,\nu) - \mathcal T(\nu,\nu) \\
    ={}& \mathcal T(\mu,\mu-\nu) + \mathcal T(\mu-\nu,\nu) = \mathcal T^{(\ell)}_\mu(\mu-\nu) + \mathcal T^{(r)}_{\nu}(\mu-\nu) \\
    ={}& {\mathcal T^{(\ell)}_\mu(\mu)-\mathcal T^{(\ell)}_\mu(\nu) + \mathcal T^{(r)}_{\nu}(\mu) - \mathcal T^{(r)}_\nu(\nu).}
  \end{align*}
    (See~\eqref{eq:def_Rlr_inf} for the definition of the operators $T^{(\ell)}_\mu$ and $T^{(r)}_{\nu}$.)
  Therefore,
  \begin{align*}
    \norm{\tilde{\mathcal T}(\mu) - \tilde{\mathcal T}(\nu)}{\mathrm{TV}} 
    ={} & {\norm{\mathcal T^{(\ell)}_\mu(\mu)-\mathcal T^{(\ell)}_\mu(\nu) + \mathcal T^{(r)}_{\nu}(\mu) - \mathcal T^{(r)}_\nu(\nu)}{\mathrm{TV}}} \\
    \leq{}& \norm{\mathcal T^{(\ell)}_\mu(\mu)-\mathcal T^{(\ell)}_\mu(\nu) }{\mathrm{TV}} + \norm{\mathcal T^{(r)}_{\nu}(\mu) - \mathcal T^{(r)}_\nu(\nu)}{\mathrm{TV}} \\
    \leq{}& \vvvert\mathcal T^{(\ell)}_\mu\vvvert\norm{\mu-\nu}{\mathrm{TV}} + \vvvert\mathcal T^{(r)}_\nu\vvvert\norm{\mu-\nu}{\mathrm{TV}}\\
    \leq{}& \tau_\ell(\mathcal T)\norm{\mu-\nu}{\mathrm{TV}} + \tau_r(\mathcal T)\norm{\mu-\nu}{\mathrm{TV}} \\
    ={}& q\norm{\mu-\nu}{\mathrm{TV}}.
  \end{align*}
  Therefore, $\tilde{\mathcal T}$ is a contraction on $\mathcal M(S)$ and has a unique fixed point $\pi\in\mathcal P(\mathcal S)$.
  In other words, there exists a unique probability measure $\pi$ such that $\mathcal
  T(\pi,\pi) = \pi$.

 Similarly, for all $\mu_1, \mu_2, \nu_1, \nu_2\in\mathcal M(\mathcal S)$,  
 \begin{align*}
    \mathcal T(\mu_1,\mu_2) - \mathcal T(\nu_1,\nu_2)
    ={}& \mathcal T(\mu_1,\mu_2) - \mathcal T(\mu_1,\nu_2) + \mathcal T(\mu_1,\nu_2) - \mathcal T(\nu_1,\nu_2)
    ={} \mathcal T(\mu_1,\mu_2-\nu_2) + \mathcal T(\mu_1-\nu_1,\nu_2) \\
    ={}& \mathcal T^{(\ell)}_{\mu_1}(\mu_2-\nu_2) + \mathcal T^{(r)}_{\nu_2}(\mu_1-\nu_1).
  \end{align*}
  Applying this to $\nu_1 = \nu_2 = \pi$ (the unique probability measure satisfying $\pi = \mathcal T(\pi, \pi)$),
  we get that, for all $\mu_1, \mu_2 \in\mathcal M(\mathcal S)$,
  \begin{align}
    \norm{\mathcal T(\mu_1,\mu_2) - \pi}{\mathrm{TV}} 
    ={}& \norm{\mathcal T(\mu_1,\mu_2) - \mathcal T(\pi,\pi)}{\mathrm{TV}} 
    ={}\norm{\mathcal T^{(\ell)}_{\mu_1}(\mu_2-\pi) + \mathcal T^{(r)}_{\pi}(\mu_1-\pi)}{\mathrm{TV}} \notag\\
    \leq{}& \vvvert\mathcal T^{(\ell)}_{\mu_1}\vvvert\norm{\mu_2-\pi}{\mathrm{TV}} + \vvvert\mathcal T^{(r)}_\pi\vvvert\norm{\mu_1-\pi}{\mathrm{TV}}
    \leq{} \tau_\ell(\mathcal T)\norm{\mu_2-\pi}{\mathrm{TV}} + \tau_r(\mathcal T)\norm{\mu_1-\pi}{\mathrm{TV}} \notag\\
    \leq{}& (\tau_\ell(\mathcal T) + \tau_r(\mathcal T))\max(\norm{\mu_2-\pi}{1}, \norm{\mu_1-\pi}{\mathrm{TV}}) \notag\\
    ={}& q\max(\norm{y-\pi}{\mathrm{TV}}, \norm{x-\pi}{\mathrm{TV}}),\label{eq:contraction_inf}
    \end{align}
  where we have used $q = \tau_\ell(\mathcal T) + \tau_r(\mathcal T)$.
  We now prove inductively on $n\geq 0$, that for all $n\geq 0$, for all $i\geq 1$,
  \begin{equation}\label{eq:induction_hyp_inf}
  \norm{\pi_i^{\sss (n)}-\pi}{\mathrm{TV}} \leq q^n \max_{j\geq 1} \norm{\pi_j^{\sss (0)} - \pi}{\mathrm{TV}}.
\end{equation}
  The base case holds straightforwardly; 
  for the induction, note that, for all $n\geq 1$, for all $i\geq 1$,
  because $\pi = \mathcal T(\pi, \pi)$ and~\eqref{eq:rec_laws_inf}, we have
  \begin{align*}
    \norm{\pi_i^{\sss (n)}-\pi}{\mathrm{TV}} 
    ={}& \norm{\mathcal T\big(\pi_{2i}^{\sss (n-1)},\pi_{2i+1}^{\sss (n-1)}\big)-\mathcal T(\pi,\pi)}{\mathrm{TV}} \\
    \leq{}& q \max \big( \norm{\pi_{2i}^{\sss (n-1)}-\pi}{\mathrm{TV}}, \norm{\pi_{2i+1}^{\sss (n-1)}-\pi}{\mathrm{TV}} \big),
    \end{align*}
    by~\eqref{eq:contraction_inf}.
    Now, by the induction hypothesis, we get
    \begin{align*}
    \norm{\pi_i^{\sss (n)}-\pi}{\mathrm{TV}}
    \leq{}& q \max \big( q^{n-1} \max_{j\geq 1} \norm{\pi_j^{\sss (0)}-\pi}{\mathrm{TV}}, q^{n-1}\max_{j\geq 1}\norm{\pi_j^{\sss (0)}-\pi}{\mathrm{TV}} \big) \\
    \leq{}& q^n \max_{j\geq 1}\norm{\pi_j^{\sss (0)}-\pi}{\mathrm{TV}},
  \end{align*}
  which concludes the proof of~\eqref{eq:induction_hyp_inf}. The result follows by taking $i=1$ in~\eqref{eq:induction_hyp_inf}.
\end{proof}

Note that Lemma~\ref{lem:intralayer-indep} holds in the infinitely-many colour case because the proof was independent of the set of colours.

\begin{lemma}\label{lem:link_MC_inf}
Fix $n\geq 0$; we work conditionally on $G_{n_1(n)}$ and $(X(i))_{1\leq i\leq n_1(n)}$, and on the event $\mathcal E_n\cap\mathcal F_n$.
We order the nodes in $L(\ell(n)-1)$ from left to right so that $H'_n$ is a plane binary tree.
For all $1\leq i\leq 2^{\ell(n)-1}$, we let $\pi_i^{\sss (0)}$ be the distribution of the label of the $i$-th node in level $L(\ell(n)-1)$ (from left to right); if $i>2^{\ell(n)-1}$, we let $\pi_i^{\sss (0)}$ be a fixed, and arbitrary distribution on $\mathcal C^2$.

Then, in distribution, $X(n) = W_1^{\sss (n)}$, where $(W_i^{\sss (n)})_{n\geq 0, i\geq 1}$ is the 2-dependent Markov chain on $\mathcal
C^2$ with operator $\mathcal T$ given by~\eqref{eq:def_T_inf}.
\end{lemma}

\begin{proof}
This is straightforward from Lemma~\ref{lem:intralayer-indep} (which gives independence of the labels of the leaves - needed in the
definition of a 2-dependent Markov chain), and from the definition of the $\mathcal R$-labelled DAG, which gives the Markov operator
$\mathcal T$.
\end{proof}

Proposition~\ref{lem:tensor-props_inf} states that under Assumptions \textbf{(B)} and \textbf{(E)}, there exists a unique $\pi\in \mathcal
P(\mathcal C)$ such that $\mathcal T(\pi, \pi) = \pi$, and $W_1^{\sss (n)}$ converges in distribution to $\pi$.
We thus get, by Lemma~\ref{lem:link_MC_inf}, that $X(n)\Rightarrow \pi$ in distribution as $n\uparrow\infty$, conditionally on $G_{n_1(n)}$
and $(X(i))_{1\leq i\leq n_1(n)}$, and on the event $\mathcal E_n\cap\mathcal F_n$.

This means that, for all measurable sets $A$ of $\mathcal C \times \mathcal C$,
\[\frac1{\mathbb P(\mathcal E_n\cap \mathcal F_n)}\cdot\mathbb P\big(\{X(n) \in A\}\cap \mathcal E_n\cap \mathcal F_n \mid G_{n_1(n)}, (X(i))_{1\leq i\leq n_1(n)}\big) \to \pi(A),\]
as $n\uparrow\infty$.

As in~\cref{sub:proof}, we get
\[\mathbb P\big(X(n) \in A \big) \to \pi(A).\]

Hence, to prove Theorem~\ref{thm:convergence_inf}, it only remains to prove that there is a unique $\nu\in \mathcal P(\mathcal C)$ such that $\sigma\nu = \mathcal R(\nu, \nu)$, and that $\pi= \nu\otimes \nu$.

By definition of $\mathcal T$ (see~\eqref{eq:def_T_inf}) and by Lemma~\ref{lem:xin_inf}, for any $A_1, A_2$ two measurable subsets of $\mathcal C$, for any $n\geq 1$,
\begin{equation}\label{eq:calc_boxes}
\mathcal T(\pi, \pi)(A_1\times A_2)
= \frac1{\sigma^2}\int_{(\alpha_1, \alpha_2)\in\mathcal C^2} \rho_{\alpha_1,\alpha_2}(A_1)\mathrm d\pi(\alpha_1,\alpha_2)
\int_{(\beta_1, \beta_2)\in\mathcal C^2} \rho_{\beta_1, \beta_2}(A_2)\mathrm d\pi(\beta_1, \beta_2)
= \nu(A_1)\nu(A_2),
\end{equation}
where we have set $\nu(A) =  \frac1{\sigma}\int_{\alpha_1, \alpha_2} \rho_{\alpha_1,\alpha_2}(A)\mathrm d\pi(\alpha_1,\alpha_2)$, for all $A\in\mathcal C$; in other words,
\begin{equation}\label{eq:def_nu_inf}
\nu =  \frac1{\sigma}\int_{\alpha_1, \alpha_2} \rho_{\alpha_1,\alpha_2}\mathrm d\pi(\alpha_1,\alpha_2)
\end{equation} 
Because the product topology on $\mathcal C\times\mathcal C$ is generated by ``boxes'' (subsets of the form $A_1\times A_2$ as above), and because the set of all boxes is stable by intersection (and thus a $\pi$-system), 
this implies that $\mathcal T(\pi, \pi)(A) = (\nu\otimes \nu)(A)$ for any measurable subset of $\mathcal C\times\mathcal C$. We thus get that $\pi =\mathcal T(\pi, \pi)= \nu\otimes\nu$.

Note that, by definition (see~\eqref{eq:def_nu_inf} for the definition of $\nu$, and~\eqref{eq:def_R_inf} for the definition of $\mathcal R$), we have $\sigma \nu = \mathcal R(\nu, \nu)$. 
It only remains to show that this solution is unique. 
Assume that there exists $\nu'\in\mathcal P(\mathcal C)$ such that $\sigma \nu' = \mathcal R(\nu', \nu')$.
Then the same calculation as in~\eqref{eq:calc_boxes} shows that $\mathcal T(\nu'\otimes \nu', \nu'\otimes\nu') = \nu'\otimes \nu'$. Because $\pi = \nu\otimes\nu$ is the unique solution of $\mathcal T(\pi, \pi) =\pi$ in $\mathcal P(\mathcal C\times \mathcal C)$, we get that $\pi = \nu'\otimes\nu'$ and thus that $\nu' = \nu$, proving uniqueness of the solution of~\eqref{eq:fix_point_R_inf}.

\bibliographystyle{alpha}
\bibliography{references}

\end{document}